\documentclass[a4paper,12pt,reqno]{amsart}
\usepackage{amsfonts,amssymb,amscd,amsmath,latexsym,amsbsy,enumerate,stmaryrd,a4wide}
\usepackage[hypertex]{hyperref}

\newtheorem{thm}{Theorem}[section]
\newtheorem{cor}[thm]{Corollary}
\newtheorem{lemma}[thm]{Lemma}
\newtheorem{prop}[thm]{Proposition}

\theoremstyle{definition}

\newtheorem{cond}[thm]{Condition}

\numberwithin{equation}{section}

\def\al{\alpha}
\def\be{\beta}
\def\ga{\gamma}
\def\de{\delta}
\def\ep{\varepsilon}

\def\te{\theta}
\def\la{\lambda}

\def\si{\sigma}
\def\vp{\varphi}

\def\Z{\mathbb{Z}}

\def\R{\mathbb{R}}
\def\C{\mathbb{C}}
\def\N{\mathbb{N}}

\def\hf{\frac12}

\def\cC{\mathcal C}
\def\cD{\mathcal D}
\def\cF{\mathcal F}
\def\cH{\mathcal H}
\def\cK{\mathcal K}
\def\cO{\mathcal O}

\newcommand{\Res}[1]{\underset{#1}{\mathrm{Res}}\,}
\newcommand{\rphis}[5]{\,_{#1}\vp_{#2} \left( \genfrac{.}{.}{0pt}{}{#3}{#4}
\ ;#5 \right)}

\begin{document}
\title[Indeterminate moment problem]
{The indeterminate moment problem \\ for the $q$-Meixner polynomials}

\author{Wolter Groenevelt and Erik Koelink}
\address{Technische Universiteit Delft, DIAM, PO Box 5031,
2600 GA Delft, the Netherlands}
\email{w.g.m.groenevelt@tudelft.nl}
\address{Radboud Universiteit, IMAPP, FNWI, Heyendaalseweg 135, 6525 AJ Nijmegen,
the Netherlands}
\email{e.koelink@math.ru.nl}

\begin{abstract} For a class of orthogonal polynomials related to the
$q$-Meixner polynomials corresponding to an indeterminate moment problem
we give a one-parameter family of orthogonality measures. For these
measures we complement the orthogonal polynomials to an orthogonal basis
for the corresponding weighted $L^2$-space explicitly. The result is
proved in two ways; by a spectral decomposition of a suitable operator
and by direct series manipulation.
We discuss extensions to
explicit non-positive
measures and the relation to other indeterminate moment problems for
the continuous $q^{-1}$-Hahn and $q$-Laguerre polynomials.
\end{abstract}

\keywords{Indeterminate moment problem, orthogonal polynomials, $q$-Meixner polynomials}

\maketitle


\section{Introduction}\label{sec:introduction}

Stieltjes \cite{Stie} introduced and studied indeterminate moment problems on the
half-line in connection with continued fractions. Since  Stieltjes' work the study of the
moment problem has flourished and we refer to the book by Akhiezer \cite{Akhi} for
more information as well as to Kjeldsen \cite{Kjel} for an historic overview.
In this paper we study an indeterminate moment problem related to the
$q$-Meixner polynomials, which can be considered as an extension of Stieltjes'
example of an indeterminate moment problem related to the Stieltjes-Wigert
polynomials via the $q$-Laguerre polynomials, see the scheme \cite[p.24]{Chri}.
We give a one-parameter family of orthogonality measures whose support is
contained in the half-line $[-1,\infty)$
such that the $q$-Meixner polynomials are orthogonal with respect to these
measures, see Proposition \ref{prop:orthogpolMeixner}. Note however that the
$q$-Meixner polynomials considered here are relabeled $q$-Meixner polynomials as in e.g. \cite{KoekS},
but the conditions on the parameters for orthogonality are mutually exclusive.
From the general theory for the moment problem \cite{Akhi} it is
known that the polynomials are not dense in the corresponding weighted
$L^2$-spaces, and we give an explicit basis for the weighted $L^2$-space
complementing the orthogonal polynomials and the precise result is given in
Corollary \ref{cor:thmLDselfadjointandspectraldecomp}.
We present two proofs of the result. The first proof is based on a spectral
decomposition of a suitable $q$-difference operator $L$, and this proof is
presented in Section \ref{sec:spectraldecomposition}. The second proof
consists of a direct proof, based solely on basic hypergeometric series,
and is given in Section \ref{sec:directproofs}. We discuss an extension to
non-positive orthogonality measures with support not contained in any half-line,
but in this case we do not have completeness statements.

The indeterminate moment problem considered in this paper fits into the
$q$-Askey scheme, and these have been studied by Christiansen \cite{Chri}.
It nearly fits in the $q$-Meixner scheme of \cite{Chri}, but the conditions
on the parameters are different. We discuss some related limit transitions
motivated by the scheme \cite[p.24]{Chri} in Section \ref{sec:limittransitions}.

The method of proof using a spectral decomposition of a suitable
$q$-difference operator is based on the fact that these indeterminate
moment problems are related to orthogonal polynomials in
the $q$-Askey scheme \cite{KoekS}, so that they are also eigenfunctions
to an explicit difference operator. This approach has been used
successfully for indeterminate moment problems for the case
of continuous $q^{-1}$-Hahn polynomials \cite{KoelS},
$q$-Laguerre polynomials \cite{CiccKK}, Stieltjes-Wigert polynomials
\cite{ChriK-JAT}, symmetric Al-Salam--Chihara polynomials \cite{ChriK-CA}.
Usually, the difference operator is well-known, but there are problems
in determining on which Hilbert space of functions the operator
should act. As it turns out, the papers \cite{CiccKK} and \cite{KoelS}
were guided by a suitable interpretation using a quantum group
analogue of $SU(1,1)$. In case of \cite{CiccKK} the interpretation was
related to the spectral decomposition of a suitable element in a
representation of a non-commutative Hopf algebra, and in case of \cite{KoelS} it is related to
the decomposition of the analogue of the Casimir operator. In this paper
the motivation comes again from this quantum group, and we actually
give two new proofs of the self-dual orthogonality relations \cite[Thm.~6.14]{GroeKK} that
arise from the unitarity of the principal unitary series representations of the
quantum group analogue of the normaliser of $SU(1,1)$ in $SL(2,\C)$.
In the group case the orthogonality relations correspond to the
unitarity of the principal unitary series representations of $SU(1,1)$ are the
orthogonality relations of the Meixner-Krawtchouk functions, see
 \cite[\S 6.8.4]{VileK}.
We prove the result of \cite{GroeKK} in a more general setting, since we can also
easily write down more solutions to the moment problem, of which, however,
some are no longer positive.

The contents of the paper are as follows. In Section \ref{sec:qintegralavaluation} we give
a direct proof of the various orthogonality measures for the $q$-Meixner
polynomials and of a related indeterminate moment problem with only finitely many moments.
This is an easy exercise in basic hypergeometric series. In Section \ref{sec:spectraldecomposition}
we present the first proof, whose main results are stated in Theorem \ref{thm:LDselfadjointandspectraldecomp}
and its Corollary \ref{cor:thmLDselfadjointandspectraldecomp} which states the result on the
level of special functions. In Section \ref{sec:directproofs} we present a direct proof of
Corollary \ref{cor:thmLDselfadjointandspectraldecomp}, which actually extends it to a somewhat more
general set of parameters.
In Section \ref{sec:orthorelsonR} we present some direct and indirect proofs related to the
non-positive measures solving the moment problem. Finally, in Section \ref{sec:limittransitions} we discuss briefly relations with other indeterminate moment problems.

\emph{Notation.}
Throughout this paper we assume that $q \in (0,1)$ is fixed. We use standard notations for $q$-shifted factorials, $\te$-functions and basic hypergeometric series from the book by Gasper and Rahman \cite{GaspR}. For $x \in \C$ and $n \in \N \cup \{\infty\}$, $\N=\{0,1,2,3\cdots\}$, the $q$-shifted factorial $(x;q)_n$ is defined by
$(x;q)_n = \prod_{k=0}^{n-1} (1-xq^k)$,
and for $x\not= 0$ the (normalized) Jacobi $\te$-function $\te(x)$ is defined by
$\te(x) = (x,q/x;q)_\infty$.
For products of $q$-shifted factorials and products of $\te$-functions we use the notations
\[
(x_1,x_2,\ldots, x_k;q)_n =\prod_{j=1}^k (x_j;q)_n, \qquad \te(x_1,x_2,\ldots,x_k)=\prod_{j=1}^k \te(x_j).
\]
The basic hypergeometric series $_r\varphi_s$ is defined by
\[
\rphis{r}{s}{x_1,x_2,\ldots, x_r}{y_1,y_2,\ldots,y_s}{q,z} = \sum_{k=0}^\infty \frac{ (x_1,x_2,\ldots,x_r;q)_k }{ (q,y_1,y_2,\ldots,y_s;q)_k } \Big((-1)^k q^{k(k-1)/2} \Big)^{1+s-r} z^k.
\]
From this definition of the $\te$-function it follows that
$\te(x)=\te(q/x)$, $\te(x) =-x\te(qx)$, $\te(x) = -x \te(1/x)$.
We often use these identities without mentioning them. Iterating the second identity gives the $\te$-product identity
\begin{equation} \label{eq:thetaproduct}
\te(xq^k) = (-x)^{-k} q^{-\hf k(k-1)} \te(x), \qquad k \in \Z.
\end{equation}


\section{$q$-integral evaluation and orthogonal polynomials}\label{sec:qintegralavaluation}

In this section we give elementary proofs of some orthogonality relations on the
polynomial level. One for measures with only a finite number of moments, and one
with all moments. The last one corresponds obviously to an indeterminate moment
problem, which is studied in this paper.

The Jackson $q$-integral is defined by
\begin{gather*}
\int_{0}^\al f(x) \, d_qx =  (1-q) \sum_{k=0}^\infty f(\al q^k)\al q^k,\\
\int_\al^\be f(x) \, d_qx =  \int_0^\be f(x) \, d_qx - \int_0^\al f(x) \, d_qx,\\
\int_0^{\infty(\al)} f(x) \, d_qx = (1-q) \sum_{k=-\infty}^\infty f(\al q^k) \al q^k, \\
\int_\be^{\infty(\al)} f(x) \, d_qx = \int_0^{\infty(\al)} f(x) \, d_qx\, +\,
\int_{0}^\be f(x) \, d_qx, \\
\int_{\infty(\be)}^{\infty(\al)} f(x) \, d_qx = \int_0^{\infty(\al)} f(x) \, d_qx\, - \,
\int_0^{\infty(\be)} f(x) \, d_qx,
\end{gather*}
for $\al,\be \in \C\setminus \{0\}$, and $f$ is a function such that the sums converge absolutely,
see \cite[Ch.~1]{GaspR}. Note that $\int_0^{\infty(\al)} f(x) \, d_qx$  is $q$-periodic
in $\al$, and similarly we have that $\int_{\infty(\be)}^{\infty(\al)} f(x) \, d_qx$
is $q$-periodic in both $\al$ and $\be$.
In case $\al=\be q^l$ for $l\in \N$ we consider the
$q$-integral
\begin{equation*}
\int_{\be q^l}^\be f(x) d_qx = (1-q) \sum_{k=0}^{l-1}  f(\be q^k)\be q^k
\end{equation*}
as a finite sum.

\begin{lemma}\label{lem:full2psi2sum} For $|c/ab| <1$ we have
\begin{equation*}
\int_{\infty(t_-)}^{\infty(t_+)}
\frac{ (-qx, -cx;q)_\infty}{(-ax, -bx;q)_\infty}\, d_qx\, = \,
(1-q)\, t_+\,   \frac{(q, c/a, c/b;q)_\infty}{(a,b, c/ab;q)_\infty}
\frac{\te(abt_-t_+, t_-/t_+, a, b)}{\te(-at_+, -at_-,-bt_+,-bt_-)}
\end{equation*}
where $t_\pm\in\C$ so that the denominator of the integrand
has no zeroes at $t_\pm q^\Z$.
\end{lemma}

We are only interested in the case $t_-<0$, $t_+>0$, which we assume
from now on. Using \eqref{eq:thetaproduct} we can check that the right hand side
is indeed $q$-periodic in $t_-$ and $t_+$.

Lemma \ref{lem:full2psi2sum} is a just a reformulation of the ${}_2\psi_2$-summation
formula given in \cite[Exerc. 5.10]{GaspR} (with the correction that
$e/ab$ and $q^2f/e$ in the numerator on the left hand side have to be
replaced by $c/qf$ and $q^2f/c$). Note that by fixing $t_-=-1$ we see
that term $(-qx;q)_\infty$ gives zero for $x\in -q^{-\N-1}$ so that this case
leads to
\begin{equation}\label{eq:t-=-1oflemmafull2psi2sum}
\int_{-1}^{\infty(t)}
\frac{ (-qx, -cx;q)_\infty}{(-ax, -bx;q)_\infty}\, d_qx\, = \,
(1-q)\, t\,   \frac{(q, c/a, c/b;q)_\infty}{(a,b, c/ab;q)_\infty}
\frac{\te(-abt, -1/t)}{\te(-at,-bt)}
\end{equation}
for $|c/ab| <1$,
where $t=t_+$ and so that the denominator of the integrand
has no zeroes at $t q^\Z$ and at $-q^\N$.
Note that in the special case $c=-q^{-r}/t$ the numerator is zero at the
points $x=tq^k$, $k<r$, so that it is actually a $q$-integral of the
form $\int_{-1}^{tq^k}$ which can be proved directly using the
non-terminating $q$-Vandermonde summation \cite[(II.23)]{GaspR}. In this
case the restriction as in Lemma \ref{lem:full2psi2sum} is no longer
required, and we are in the case of the orthogonality measure for the
big $q$-Jacobi polynomials, see e.g. \cite[Ch.~7]{GaspR}, \cite{KoekS}.

The special case $c=0$ of Lemma \ref{lem:full2psi2sum} and
\eqref{eq:t-=-1oflemmafull2psi2sum} gives
\begin{equation}\label{eq:c=0oflemmafull2psi2sum}
\int_{\infty(t_-)}^{\infty(t_+)}
\frac{ (-qx ;q)_\infty}{(-ax, -bx;q)_\infty}\, d_qx\, = \,
(1-q)\, t_+\,   \frac{(q ;q)_\infty}{(a,b ;q)_\infty}
\frac{\te(abt_-t_+, t_-/t_+, a,  b)}{\te(-at_+, -at_-,-bt_+,-bt_-)}
\end{equation}
and
\begin{equation}\label{eq:c=0oft-=-1oflemmafull2psi2sum}
\int_{-1}^{\infty(t)}
\frac{ (-qx ;q)_\infty}{(-ax, -bx;q)_\infty}\, d_qx\, = \,
(1-q)\, t\,   \frac{(q;q)_\infty}{(a,b;q)_\infty}
\frac{\te(-abt, -1/t)}{\te(-at,-bt)}
\end{equation}
again assuming the denominator of the integrand
has no zeroes.

The restriction $t_-<0$, $t_+>0$ leads to a discrete
measure with infinite support on the real line $\R$, which
has finitely many moments in case of Lemma \ref{lem:full2psi2sum}
and \eqref{eq:t-=-1oflemmafull2psi2sum}, and where all moments
exist
in case of \eqref{eq:c=0oflemmafull2psi2sum} and
\eqref{eq:c=0oft-=-1oflemmafull2psi2sum}. It is now
straightforward to determine the corresponding orthogonal
polynomials.

\begin{prop}\label{prop:finiteorthogpol}
Define the polynomial
\[
P_n(x; a,b,c; q) \, = \, b^{-n} (b,qb/c;q)_n\, \rphis{3}{2}{q^{-n}, abq/c, -bx}{b, qb/c}{q, q},
\]
then
\begin{gather*}
\int_{\infty(t_-)}^{\infty(t_+)} P_n(x;a,b,c)\, P_m(x;a,b,c)\,
\frac{ (-qx, -cx;q)_\infty}{(-ax, -bx;q)_\infty}\, d_qx\, = \,
\de_{n,m} H_n(a,b,c) I(a,b,c;t_-,t_+), \\
H_n(a,b,c)\, = \, (-c)^{-n} q^{\hf n(n+1)} \frac{(abq^n/c;q)_n}{(abq/c;q)_{2n}}
(a,b, aq/c, bq/c;q)_n
\end{gather*}
for $|c/ab|<q^{n+m}$, and this is in particular true in case
$t_-=-1$, cf. \eqref{eq:t-=-1oflemmafull2psi2sum}. Here
$I(a,b,c;t_-,t_+)$ is the right hand side of the integral in Lemma \ref{lem:full2psi2sum}.
\end{prop}

The normalization is chosen so that $P_n(x;a,b, c;q)$ is symmetric in $a$ and $b$,
which can be proved directly using \cite[(3.2.2), (3.2.5)]{GaspR}. The polynomials
are related to the big $q$-Jacobi polynomials, see \cite[\S 7.3]{GaspR}, \cite{KoekS},
but the range of the parameters does not fit the conditions for orthogonality of the big $q$-Jacobi
polynomials.

In case $t_-=-1$ the result is contained in \cite[\S 8]{KoelS} for the part of the discrete
spectrum under the additional assumption $0<a,b<1$, $c<a$, corresponding to the first part
of $S$ in \cite[(8.1)]{KoelS}. In case of arbitrary $t_-<0$ the weight fits into the results
of
\cite{Groe}, and we also find a finite set of orthogonal polynomials, see \cite[\S3-4]{Groe},
which also gives conditions on the parameters for the weight function to be non-negative.
In light of these remarks one can consider these polynomials as $q$-analogues of the
Routh (or Romanovsky) polynomials, see \cite[\S 20.1]{Isma}. Note also that \cite{KoelS}
and \cite{Groe} actually contain different proofs of respectively \eqref{eq:t-=-1oflemmafull2psi2sum}
and Lemma \ref{lem:full2psi2sum} for the restricted parameter sets.

\begin{prop}\label{prop:orthogpolMeixner}
Define the polynomial
\[
m_n(x) \, = \, m_n(x; a,b;q) \, = \, \frac{1}{(a;q)_n}\, \rphis{2}{1}{q^{-n}, -bx}{b}{q, aq^n},
\]
then
\begin{gather*}
\int_{\infty(t_-)}^{\infty(t_+)} m_n(x)\, m_m(x)\,
\frac{ (-qx;q)_\infty}{(-ax, -bx;q)_\infty}\, d_qx\, = \,
\de_{n,m} h_n(a,b) I(a,b;t_-,t_+) \\
h_n(a,b)\, = \, \frac{q^{-n}\, (q;q)_n}{(a,b;q)_n}
\end{gather*}
and this is in particular true in case
$t_-=-1$;
\[
\int_{-1}^{\infty(t)} m_n(x)\, m_m(x)\,
\frac{ (-qx;q)_\infty}{(-ax, -bx;q)_\infty}\, d_qx\, = \,
\de_{n,m} h_n(a,b) I(a,b;t).
\]
Here $I(a,b;t_-,t_+)$, respectively $I(a,b;t)$, denote
the right hand side of \eqref{eq:c=0oflemmafull2psi2sum},
respectively \eqref{eq:c=0oft-=-1oflemmafull2psi2sum}.
\end{prop}

We have defined $m_n$ in such a way that it is symmetric in $a$ and $b$ by
\cite[(III.2)]{GaspR}.

Note that Proposition \ref{prop:finiteorthogpol} deals with
orthogonality for only a finite number of polynomials, whereas
Proposition \ref{prop:orthogpolMeixner} deals with orthogonal
polynomials. Since the polynomials $m_n$ are independent of
$t_\pm$, we see that we have an indeterminate moment problem
in case we have positivity of the measures involved, see
Condition \ref{cond:onparameters}.

The $q$-Meixner polynomials are defined by
\begin{equation}\label{eq:defqMeixnerpols}
M_n(x;b,c;q)\, = \, \rphis{2}{1}{q^{-n}, x}{bq}{q, -\frac{q^{n+1}}{c}},
\end{equation}
see \cite{GaspR}, \cite{KoekS}, which are orthogonal on the set $q^{-\N}$
for $0<b<q^{-1}$, $c>0$. It follows that
\begin{equation}\label{eq:linktoqMeixnerpols}
m_n(x;a,b;q) \, = \, \frac{1}{(a;q)_n}\, M_n(-bx; \frac{b}{q}, -\frac{q}{a}; q)
\end{equation}
and we note that the conditions for parameters of the $q$-Meixner polynomials
translate to $0<b<1$, $a<0$ which does not fit Condition \ref{cond:onparameters}.

\begin{proof}[Proof of Proposition \ref{prop:finiteorthogpol}]
Observe that for $k,l\in\N$ with $|c/ab| <q^{l+k}$ we have
\[
\int_{\infty(t_-)}^{\infty(t_+)}  (-ax;q)_l (-bx;q)_k
\frac{ (-qx, -cx;q)_\infty}{(-ax, -bx;q)_\infty}\, d_qx\, = \,
I(aq^l,bq^k,c;t_-,t_+)
\]
by Lemma \ref{lem:full2psi2sum}. A straightforward calculation
using \eqref{eq:thetaproduct}
gives
\[
\frac{I(aq^l,bq^k,c;t_-,t_+)}{I(a,b,c;t_-,t_+)} \, = \,
\frac{(a, q^{-l}c/a;q)_l}{b^l\, (cq^{-l}/ab;q)_l}
\, \frac{(b, bq/c;q)_k}{(abq^{1+l}/c;q)_k}
\]
so that
\begin{equation*}
\begin{split}
&\int_{\infty(t_-)}^{\infty(t_+)} P_n(x;a,b,c;q)\, (-ax;q)_l\,
\frac{ (-qx, -cx;q)_\infty}{(-ax, -bx;q)_\infty}\, d_qx
\, = \, \\ & I(a,b,c;t_-,t_+)\, b^{-n} (b,qb/c;q)_n\,
 \frac{(a, q^{-l}c/a;q)_l}{b^l\, (cq^{-l}/ab;q)_l}\,
\rphis{2}{1}{q^{-n}, abq/c}{abq^{l+1}/c}{q, q} \,
\end{split}
\end{equation*}
The ${}_2\vp_1$-series can be evaluated by the $q$-Vandermonde
summation formula \cite[(1.5.3)]{GaspR}, giving
$\frac{(q^{l+1-n};q)_n}{(abq^{l+1}/c;q)_n} (\frac{abq^n}{c})^n$
which equals zero for $l<n$. This proves Proposition
\ref{prop:finiteorthogpol} in case $m<n$. The case
$m=n$ follows by taking into account the symmetry of
$P_n$ in $a$ and $b$ and the above calculation.
\end{proof}

\begin{proof}[Proof of Proposition \ref{prop:orthogpolMeixner}]
The proof copies the proof of Proposition \ref{prop:finiteorthogpol}
in the case $c\to 0$ and using the $q$-binomial theorem \cite[(II.4)]{GaspR} instead
of the $q$-Vandermonde summation.
\end{proof}

Note that Proposition \ref{prop:orthogpolMeixner} concerns a
set of orthogonal polynomials for each degree. In case $t_+>0$, $t_-<0$ we
want to see which conditions on $a$ and $b$ lead to a positive measure.
We see that the weight function is positive on $t_+q^\Z$ in case
$a=\bar b$ (assuming $a\in \C\setminus \R$), or $a>0$, $b>0$ or
$a<0$, $b<0$  so that there exists $k_0\in\Z$ with
$q^{k_0}<-at_+< q^{k_0-1}$, $q^{k_0}<-bt_+< q^{k_0-1}$. However,
for general $t_-<0$ it is not possible to have a positive
weight function on $t_-q^\Z$ for the conditions mentioned. In case
$t_-=-1$, however we only have to deal with the positivity of
$\frac{(q^{k+1};q)_\infty}{(aq^k,bq^;q)_\infty}$ for $k\in\N$.
This is the case for $a=\bar b$ (assuming $a\in \C\setminus \R$), or
$a<1$, $b<1$ or if there exists $k_0\in -\N$ with
$q^{k_0}<a< q^{k_0-1}$, $q^{k_0}<b< q^{k_0-1}$.

\begin{cond}\label{cond:onparameters} $t=t_+>0$, $t_-=-1$ and
one of the following conditions on $a$ and $b$ holds:
\begin{enumerate}[(i)]
\item $a = \bar b$, with $a\in \C\setminus \R$;
\item $0<a<1$ and $0<b<1$;
\item for some $k_0\in -\N$ with
$q^{k_0}<a< q^{k_0-1}$ and $q^{k_0}<b< q^{k_0-1}$;
\item for some $k_0\in \Z$ with $q^{k_0}<-at< q^{k_0-1}$ and  $q^{k_0}<-bt< q^{k_0-1}$.
\end{enumerate}
\end{cond}

Condition \ref{cond:onparameters} ensures that the measure in \eqref{eq:c=0oft-=-1oflemmafull2psi2sum}
is non-negative, so that the polynomials $m_n$ are orthogonal with respect to a
positive measure with support contained in  $[-1,\infty)$.
Note that for fixed $t$ these four cases are mutual exclusive. In case (i),
$a,b$ are non-real, in case (ii) and (iii) $a$ and $b$ are positive and in
case (iv) $a$ and $b$ are negative.
Since $t$ is arbitrary, we see that we have an indeterminate
moment problem in the $q$-Askey scheme, and in Christiansen's classification
this fits in the $q$-Meixner class, see \cite[p.~25ff]{Chri}. Note that
from general considerations \cite{Akhi} the polynomials are not dense
in the corresponding weighted $L^2$-space.
We study the case of positive measure
in more detail in Section \ref{sec:spectraldecomposition}, and we
give an alternative direct proof in Section \ref{sec:directproofs}.


\section{Spectral decomposition}\label{sec:spectraldecomposition}

In this section we introduce an operator $L$ which is self-adjoint
for the weighted $L^2$-space corresponding to weight given by
\eqref{eq:c=0oft-=-1oflemmafull2psi2sum} under the positivity
condition of Condition \ref{cond:onparameters}. For this we
follow the strategy employed in \cite{KoelS}, as well as
\cite{KoelLaredo}, so we explicitly determine the spectral decomposition
of a suitable operator having the polynomials of
Proposition \ref{prop:orthogpolMeixner} as eigenfunctions.

We assume that $a$, $b$, $t$  satisfy Condition \ref{cond:onparameters} and
we define the weight function
\begin{equation}\label{eq:defweight}
w(x)\, = \, w(x; a,b; q) \, = \, \frac{ (-qx;q)_\infty}{(-ax, -bx;q)_\infty}.
\end{equation}
We define $\cF_q$ to be the space of complex-valued functions on
$-q^\N \cup tq^\Z$ and the Hilbert space
\begin{equation}\label{eq:defHilbertspace}
\cH_t\, = \, \cH_t(a,b) \, = \, \big\lbrace f\in \cF_q \mid
\int_{-1}^{\infty(t)} |f(x)|^2\, w(x;a,b; q)\, d_qx < \infty \big\rbrace
\end{equation}
with the corresponding inner product.

Define the difference operator for $f\in \cF_q$
\begin{equation} \label{eq:defL}
\begin{split}
(Lf)(x) &\, =\,   A(x)[f(qx)-f(x)]\, +\, B(x)[f(x/q)-f(x)],  \\
A(x)&\, =\, \Big(a+\frac{1}{x}\Big)\Big(b+\frac{1}{x}\Big),
\qquad B(x)\, =\,  \frac{q}{x}\Big(1+\frac{1}{x}\Big).
\end{split}
\end{equation}
Note that $A$ and $B$ are real-valued on $\R$ for $a,b,t$ satisfying Condition
\ref{cond:onparameters}.

At this point we note that the $q$-Meixner polynomials
\eqref{eq:defqMeixnerpols} and their orthogonality relations can be determined
in the same way using the operator $L$, but for a measure supported
on $-bq^{-\N}$. In that case $L$ reduces to the standard second order
difference equation for the $q$-Meixner polynomials \cite{KoekS}.

For  $f\in \cF_q$ we define
\[
\begin{split}
&f(0^+)\, =\, \lim_{k \to \infty}f(tq^k), \qquad
f(0^-)\, =\, \lim_{k \to \infty}f(-q^k), \\
&f'(0^+)\, =\, \lim_{k \to \infty}(D_qf)(tq^k), \
f'(0^-)\, =\, \lim_{k \to \infty}(D_qf)(-q^k),
\end{split}
\]
provided the limits exists and where $D_qf(x) = \frac{f(x)-f(qx)}{(1-q)x}$, $x\not= 0$,
is the $q$-derivative \cite[Ch. 1]{GaspR}.

We can formulate the main result of this section after introducing the notation
\begin{equation} \label{eq:phi=2phi2}
\phi_\ga(x)\, =\, \phi_\ga(x;a,b;q)\, =\, \rphis{2}{2}{-1/x, -1/\ga}{a,b}{q,ab\ga x}.
\end{equation}
for the $q$-Meixner functions.

\begin{thm}\label{thm:LDselfadjointandspectraldecomp} Define
\[
\cD\, =\,  \left\{ f \in \cH_t \ | \ Lf \in \cH_t,\ f(0^-)=f(0^+),\ f'(0^-)= f'(0^+) \right\}
\, \subset\, \cH_t,
\]
then $(L, \cD)$ is self-adjoint, and  the  spectrum of $(L, \cD)$ consists of
\[
\begin{split}
\si(L)\, &=\,  {-ab} \cup \{ -ab(1-q^k) \colon k\in \N\} \cup
\{ -ab(1+q^k/abt) \colon k\in \Z\} \\
\, &= \, \mu(0) \cup \mu(-q^\N) \cup \mu(q^\Z/abt)
\end{split}
\]
where $\mu(\ga) = -ab(1+\ga)$.
Moreover, $\mu(-q^\N) \cup \mu(q^\Z/abt)$ corresponds to the point spectrum
$\si_p(L)$, which is simple. The eigenvector is given by the Meixner
function $\phi_\ga(\cdot) \, = \, \phi_\ga(\cdot;a,b;q)\in \cD$, where
$L\, \phi_\ga \, = \, \mu(\ga)\, \phi_\ga$, $\ga \in -q^\N \cup q^\Z/abt$.
\end{thm}

We prove Theorem \ref{thm:LDselfadjointandspectraldecomp} in this section.
As a corollary to its proof we find the following orthogonality relations.

\begin{cor}\label{cor:thmLDselfadjointandspectraldecomp}
$\{ \phi_\ga(\cdot;a,b;q)\colon \ga \in -q^\N \cup q^\Z/abt\}$ is an
orthogonal basis for $\cH_t$, and the orthogonality relations
\begin{equation*}
\begin{split}
\int_{-1}^{\infty(t)}
\phi_\ga(x;a,b;q) \overline{\phi_\la(x;a,b;q)} \, w(x;a,b;q)\, d_qx\, &= \,
\de_{\ga,\la} \, H_\ga(a,b;q) \, I(a,b;t) \\
H_\ga(a,b;q)\, &= \, \frac{(q, -a\ga, -b\ga;q)_\infty}{|\ga|\, (a,b,-q\ga;q)_\infty}
\end{split}
\end{equation*}
where $I(a,b;t)$ is the right hand side of  \eqref{eq:c=0oft-=-1oflemmafull2psi2sum} and
$\ga,\la \in -q^\N \cup q^\Z/abt$.
\end{cor}

Corollary \ref{cor:thmLDselfadjointandspectraldecomp} gives an independent  proof of
Proposition \ref{prop:orthogpolMeixner} as well as the $q$-integral evaluation
\eqref{eq:c=0oft-=-1oflemmafull2psi2sum} as a special case for $\ga,\la\in -q^\N$,
respectively $\ga=\la=-1$, as follows from
\eqref{eq:qMeixnerfunctionpols}.
Corollary \ref{cor:thmLDselfadjointandspectraldecomp} is also proved in this
section, and a direct proof based on series manipulation is given in
Section \ref{sec:directproofs}, whereas the polynomial case corresponds
to Proposition \ref{prop:orthogpolMeixner}.

Note that Corollary \ref{cor:thmLDselfadjointandspectraldecomp} gives rise
to many solutions of the moment problem corresponding to the orthogonal
polynomials $m_n(\cdot;a,b;q)$, e.g. by varying over $t$, integrating
over $t\in (q,1]$ to get a orthogonality measure which is partially
absolutely continuous or by
multiplying the weight by a suitable $1+C^{-1}\phi_{q^k/abt}(x)>0$,
which can be done if $|\phi_{q^k/abt}(x)|\leq C$ which is the case
for $|q^k/abt|>1$ (by Lemma \ref{lem:asymptoticPhi} and \eqref{eq:defPhi}).
The results of Proposition \ref{prop:orthogpolMeixner} and
Corollary \ref{cor:thmLDselfadjointandspectraldecomp} do not fit
precisely in the $q$-Meixner tableau of the indeterminate $q$-Askey
scheme, see \cite{Chri}, but the Krein parametrization for this case should
follow analogously. It is not clear how to proceed to find the corresponding
Pick function for this solution to the moment problem.

With Condition \ref{cond:onparameters}(i) and $t\in q^\Z$
this corresponds to \cite[Thm.~6.14]{GroeKK}, where $a$ and $b$ are
related to the label of the unitary principal series representation. Since the result corresponds
to the unitarity of the unitary principal series representations, we may view
Corollary \ref{cor:thmLDselfadjointandspectraldecomp} as a $q$-analogue
of the Krawtchouk-Meixner functions, see \cite[\S 6.8.4]{VileK}.

The orthogonality relations of Corollary \ref{cor:thmLDselfadjointandspectraldecomp}
are self-dual, as follows from the fact that $H_\ga$ is essentially $(|\ga|w(\ga))^{-1}$.


\subsection{Self-adjointness}\label{ssec:selfadjointness}

Since the operator $L$ is an unbounded operator on $\cH_t$, we
need to describe a suitable domain. This is described in
Proposition \ref{prop:Lisselfadjoint}.

We consider the truncated inner product
for $l\in \N$, $m,n \in \Z$
\begin{equation}\label{eq:truncatedinnerporduct}
\langle f, g \rangle_{l;m,n} =
 \int_{-1}^{-q^{l+1}} f(x) \overline{g(x)} \, w(x)\, d_qx +
\int_{t_+q^{m+1}}^{t_+q^n} f(x) \overline{g(x)}\, w(x)\, d_qx
\end{equation}
for arbitrary $f,g\in \cF_q$.
Recall the convention that the $q$-integrals are finite sums,
see Section \ref{sec:qintegralavaluation}.
Taking the limits $l\to\infty$ and $m\to -\infty$, $n\to\infty$
gives back the inner product in $\cH_t$ for $f,g\in \cH_t$.

For $f,g\in \cF_q$ we define the Casorati determinant (or the Wronskian) $D(f,g) \in \cF_q$ by
\begin{equation} \label{eq:D(f,g)}
\begin{split}
D(f,g)(x)\, &=\, \Big( f(x) g(qx) - f(qx) g(x)\Big) v(x)\,
=\, \Big((D_qf)(x)g(x) - f(x)(D_qg)(x)\Big)u(x),
\end{split}
\end{equation}
where $D_q$ is the $q$-derivative and
\[
v(x) = \frac{1-q}{x} \frac{ (-qx;q)_\infty }{(-aqx,-bqx;q)_\infty}, \qquad
u(x)=(1-q)x v(x).
\]

\begin{lemma} \label{lem:trinprod(f,g)=D(f,g)}
For $f,g \in \cF_q$ we have
\[
\begin{split}
\langle Lf, g \rangle_{l;m,n} &- \langle f,L g \rangle_{l;m,n}\, = \,
 D(f,\overline{g})(-q^{l})\,  +\,  D(f,\overline{g})(tq^{n-1}) \, -\, D(f,\overline{g})(tq^{m}).
\end{split}
\]
\end{lemma}

Note that Lemma \ref{lem:trinprod(f,g)=D(f,g)} in particular implies that $L$, restricted
to the finitely supported functions $\cH_t$, is a symmetric operator.

\begin{proof}
Using the real-valuedness of $A$ and $B$ on $\R$ we find
\[
\begin{split}
&\Big((Lf)(x)\overline{g(x)} - f(x)\overline{(Lg)(x)}\Big) (1-q)x w(x) \,
= \, A(x) (1-q) x w(x) \Big( f(qx)\overline{g(x)}  - f(x) \overline{g(qx)} \Big)\\
 &\qquad\qquad\qquad -\,  B(x) (1-q)x w(x) \Big( f(x)\overline{g(x/q)} - f(x/q) \overline{g(x)} \Big) \,
= \, D(f,\overline{g})(x/q) - D(f,\overline{g})(x)
\end{split}
\]
for real $x$.
Plugging this into $\langle Lf, g\rangle_{l;m,n} -\langle f, Lg\rangle_{l;m,n}$ we see
that \eqref{eq:truncatedinnerporduct} gives two finite telescoping sums leading to the result.
\end{proof}

Lemma \ref{lem:trinprod(f,g)=D(f,g)} shows that the Casorati determinant plays
an important role in determining a dense domain for $L$ such that we have
a self-adjoint operator.
We observe that
\begin{equation} \label{eq:behaviourwatzero}
w(tq^k) \, = \, 1+ \cO(q^k), \qquad w(-q^k) \, = \, 1+ \cO(q^k), \qquad k\to \infty
\end{equation}
and, using the theta-product identity \eqref{eq:thetaproduct},
\begin{equation} \label{eq:behaviourwatinfty}
w(tq^k) \, = \,
\frac{ \te(-tq) }{ \te(-at,-bt) } \left(\frac{abt}{q}\right)^k q^{\hf k(k-1)}\Bigl(1+ \cO(q^{-k})\Bigr)\qquad k\to -\infty.
\end{equation}
Using the asymptotic behaviour of the weight function we conclude that
for $f\in\cH_t$ we have
\begin{equation}\label{eq:asymptforeltsofH}
\lim_{k\to\infty} f(tq^k) q^{\hf k} \, = \, 0, \quad \lim_{k\to\infty} f(-q^k) q^{\hf k} \, = \, 0,
\quad \lim_{k\to -\infty} f(tq^k) (abt)^{\hf k} q^{\frac14 k(k-1)}\, = \, 0.
\end{equation}

\begin{lemma} \label{lem:D(f,g)(infty)}
Let $f,g \in \cH_t$, then
$\lim_{k \to -\infty}D(f,\overline{g})(tq^k) = 0$.
\end{lemma}

Lemma \ref{lem:D(f,g)(infty)} shows that we don't require a condition at $\infty$ for
the definition of the domain of $L$.

\begin{proof} Since $v(x)=(1-q)xA(x)w(x)$ we find  from
\eqref{eq:behaviourwatinfty} that
\begin{equation}\label{eq:asymptoticsv}
v(tq^k)\, = \frac{(1-q) \te(-tq) }{t\te(-atq,-btq)}\, (abt)^k q^{\hf k(k-1)}
\Big(1 + \cO(q^{-k})\Big), \qquad k \to -\infty.
\end{equation}
Hence, for $f,g \in \cH_t$ we have  by \eqref{eq:asymptforeltsofH}
\[
\begin{split}
&\lim_{k\to\infty} f(tq^k)g(tq^{k+1})v(tq^k)\, = \,
K \lim_{k \to -\infty} f(tq^k)g(tq^{k+1}) (abt)^k q^{\hf k(k-1)}\,  = \, \\
&K(abt)^{-\hf} \lim_{k \to -\infty} q^{-k/2} \Big(f(tq^k)(abt)^{k/2} q^{\frac14k(k-1)}\Big) \Big(g(tq^{k+1}) (abt)^{\hf(k+1)} q^{\frac14 k(k+1)}\Big) = 0,
\end{split}
\]
with the constant $K = \frac{(1-q)\te(-tq)}{t\te(-atq,-btq)}$, so that
$\lim_{k \to -\infty} D(f,g)(tq^{k})=0$
by \eqref{eq:D(f,g)}.
\end{proof}

Recall the definition of $\cD$ in Theorem \ref{thm:LDselfadjointandspectraldecomp},
then we see that $\cD$ is dense in $\cH_t$, since it contains the dense subspace of
finitely supported functions.

\begin{prop} \label{prop:Lisselfadjoint}
The operator $(L,\mathcal D)$ is self-adjoint.
\end{prop}

Proposition \ref{prop:Lisselfadjoint} proves the first statement of
Theorem \ref{thm:LDselfadjointandspectraldecomp}. The proof of
Proposition \ref{prop:Lisselfadjoint} is completely analogous to the proof of
\cite[Prop.~2.7]{KoelS}, and is left to the reader. Note that we can
also introduce a one-parameter family of domains $\cD_\al$ as in
\cite{KoelS} so that $(L, \cD_\al)$ is also self-adjoint. In particular,
$L$ restricted to the finitely supported functions in $\cH_t$ is not
essentially self-adjoint.

In order to find the spectral decomposition we need to find sufficiently
many eigenfunctions. The first step is the following lemma, whose
proof follows \cite[Lemma~3.1, Prop.~3.2, Cor.~3.3.]{KoelS}.

\begin{lemma}\label{lem:dimVmu}
For $\mu\in \C$ we define
\[
V_\mu \, = \, \{f\in \cF_q \mid Lf(x)\, = \, \mu f(x)\ \text{for}\
x\in -q^{\N+1}\cup tq^\Z, \ f(0^+)\,=\, f(0^-), \ f'(0^+)\,=\, f'(0^-)\}.
\]
Then $\dim V_\mu\leq 2$. Moreover, for $f_1,f_2\in V_\mu$  the Casorati
determinant $D(f_1,f_2)$ is constant as a function on $-q^{\N+1}\cup tq^\Z$.
In case $\dim V_\mu = 2$, the restriction operator from $V_\mu$ to the
space $\{f\in \cF_q \mid Lf(x)\, = \, \mu f(x)\ \text{for}\
x\in tq^\Z\}$ is a bijection.
\end{lemma}

So we don't impose the condition $Lf(x)\, = \, \mu f(x)$ at $x=-1$.


\subsection{$q$-Meixner functions}\label{ssec:qMeixnerfunctions}

It is time to study the $q$-Meixner functions \eqref{eq:phi=2phi2}
in more detail, and we take this up now.

The $q$-Meixner functions defined by \eqref{eq:phi=2phi2} are
obviously symmetric in $a$ and $b$, as well as self-dual, i.e.
symmetric in $x$ and $\ga$;
\begin{equation}\label{eq:dualitysymmetryab}
\phi_\ga(x) \, = \, \phi_x(\ga), \qquad
\phi_\ga(x;a,b;q)\, = \, \phi_\ga(x;b,a;q).
\end{equation}
Moreover, since $(-1/x;q)_n (ab\ga x)^n$ is a polynomial of degree $n$ in $x$,
it follows that $\phi_\ga(x)$ is an entire function in $x$, hence also in $\ga$.

Using transformation formulas for basic hypergeometric series we can find several more explicit expressions for the $q$-Meixner functions.
From applying \cite[(III.4)]{GaspR} with $(A,B,C,Z) = (-1/x, -b\ga, b, -ax)$ (we write the parameters $a,b,c,z$ from \cite{GaspR} in capitals in order to avoid confusion) we find
\begin{equation} \label{eq:phi=2phi1}
\phi_\ga(x) \, =\,  \frac{ (-ax;q)_\infty }{(a;q)_\infty } \rphis{2}{1}{-1/x, -b\ga}{b}{q,-ax}, \qquad |ax|<1,
\end{equation}
and applying Heine's transformation \cite[(III.2)]{GaspR} with $(A,B,C,Z) = (-1/x, -b\ga, b, -ax)$ then gives
\begin{equation} \label{eq:phi}
\phi_\ga(x)\, = \,  \frac{(ab\ga x,-1/\ga;q)_\infty}{(a,b;q)_\infty} \rphis{2}{1}{-a\ga, -b\ga}{ab\ga x}{q,-\frac{1}{\ga}}, \qquad |\ga|>1.
\end{equation}
Furthermore, applying \cite[(III.4)]{GaspR} to \eqref{eq:phi} with $(A,B,C,Z) = (-b\ga,-1/x, b, -ax)$ we find
\begin{equation} \label{eq:phi=2phi2A}
\phi_\ga(x) =  \frac{(ab\ga x;q)_\infty}{(a;q)_\infty}\rphis{2}{2}{-b\ga,-bx}{b,ab\ga x}{q,a}.
\end{equation}
Observe that the $_2\varphi_1$-series in \eqref{eq:phi=2phi1} terminates for $x \in -q^\N$, so in this case $\phi_\ga(x)$ is a polynomial in $\ga$,
and in particular $\phi_\ga(-1) = 1$.

So for $n\in \N$ we have by \eqref{eq:phi=2phi1} and \eqref{eq:dualitysymmetryab}
the reduction to Proposition \ref{prop:orthogpolMeixner};
\begin{equation}\label{eq:qMeixnerfunctionpols}
\phi_{-q^n}(x;a,b;q)\, = \, \frac{1}{(a;q)_n}\, \rphis{2}{1}{q^{-n}, -bx}{b}{q, aq^n}
\, = \, m_n(x;a,b;q).
\end{equation}

\begin{prop} \label{prop:Lphi}
The $q$-Meixner function $\phi_\ga$ satisfies $(L\phi_\ga)(x) \, =\, \mu(\ga)\phi_\ga(x)$
for $x\in \R\setminus\{0\}$, $\ga\in\C$.
\end{prop}

\begin{proof}
This follows from one of Heine's $q$-contiguous relations, see \cite[Exer.1.10(iv)]{GaspR}. Denote
\[
\varphi(C)=\rphis{2}{1}{A,B}{C}{q,Z},
\]
then
\[
\begin{split}
(q-C)(ABZ-C)\varphi(Cq^{-1})+[C(q-C)+(C(A+B)-&AB(1+q))Z]\varphi(C)\\
&+\frac{ (C-A)(C-B)Z}{1-C}\varphi(Cq)=0.
\end{split}
\]
Substitute $(A,B,C,Z) \mapsto (-a\ga,-b\ga,ab\ga x, -1/\ga)$, and multiply by $\frac{(ab\ga x,-1/\ga;q)_\infty}{(a,b;q)_\infty}$, then using \eqref{eq:phi} we find
\[
-qab\ga(1+x)\phi_\ga(x/q) + ab\ga[qx-ab\ga x^2+ax+bx+1+q]\phi_\ga(x)- ab\ga(1+bx)(1+ax)\phi_\ga(qx)=0
\]
for $|\ga|>1$. By Condition \ref{cond:onparameters} $ab\not= 0$, so we find the
result for $|\ga|>1$. Since the expression is analytic in $\ga$ the result follows.
\end{proof}

For later use we list some useful properties of the $q$-Meixner functions.

\begin{lemma} \label{lem:propertiesofphi}
The $q$-Meixner function $\phi_\ga$ has the following properties.
\begin{enumerate}[(i)]
\item $\displaystyle (D_q \phi_\ga)(x) =  \frac{ -ab(1+\ga) }{(1-q)(1-a)(1-b)} \phi_{\ga/q}(x;aq,bq;q)$
\item $\displaystyle \lim_{x \to 0} \phi_\ga(x) = \frac{1}{(a;q)_\infty}\rphis{1}{1}{-b\ga}{b}{q,a}$
\item For $ab\ga t\notin q^\Z$,
\[
\phi_\ga(tq^k) = \frac{(-1/\ga;q)_\infty\te(ab\ga t)}{(a,b;q)_\infty} (-ab\ga t)^{-k} q^{-\hf k(k-1)}\Big(1+ \cO(q^{-k})\Big), \qquad k \to -\infty.
\]
\end{enumerate}
\end{lemma}

It follows that $\phi_\ga(0^+)\, = \, \phi_\ga(0^-)$ and $\phi_\ga'(0^+)\, = \, \phi_\ga'(0^-)$
by Lemma \ref{lem:propertiesofphi}(i), (ii). However, Lemma \ref{lem:propertiesofphi}(iii) and
\eqref{eq:asymptforeltsofH} show that in general $\phi_\ga \notin\cH_t$. It remains
to investigate what happens in case the leading coefficient vanishes, i.e.
for $\ga \in -q^\N \cup q^\Z/abt$. Note that the behaviour of $\phi_\ga$ at $x\to 0$
suffices to have square integrability with respect to the weight $w$ at zero.

\begin{proof}
The proof of (i) can either be done straightforwardly using
\[
(-1/x;q)_n - (-1/qx;q)_n q^n =(1-q^n) (-1/x;q)_{n-1}.
 \]
and the expression \eqref{eq:phi=2phi2}. Or one can use the
duality \eqref{eq:dualitysymmetryab} and \eqref{eq:phi=2phi1}
and the contiguous relation \cite[Exerc. 1.9(ii)]{GaspR} to
prove the first statement.

The second statement follows immediately from \eqref{eq:phi=2phi2A}.

For the last statement we use \eqref{eq:phi=2phi2A} and some rewriting
to find for $k\to -\infty$
\begin{equation*}
\begin{split}
\phi_\ga(tq^k)\, &= \,
\frac{\te(ab\ga t)\, (-abt\ga)^{-k} q^{-\hf k(k-1)}}{(a;q)_\infty}
\sum_{l=0}^\infty \frac{(-b\ga;q)_l \ga^{-l} q^{\hf l(l-1)}}
{(q,b;q)_l} \, \frac{(-q^{-k}/bt;q^{-1})_l}{(q^{-k-l}/ab\ga t;q)_\infty} \\
&=\, \frac{\te(ab\ga tq^k)}{(a;q)_\infty}
(-abt\ga)^{-k} q^{-\hf k(k-1)}\, \rphis{1}{1}{-b\ga}{b}{q, -\frac{1}{\ga}}
\Bigl( 1 + \cO(q^{-k})\Bigr)
\end{split}
\end{equation*}
using the theta-product identity \eqref{eq:thetaproduct} and
dominated convergence. The ${}_1\vp_1$-summation formula
\cite[(II.5)]{GaspR} gives the result.
\end{proof}

We can characterize the solution $\phi_\ga$ to the eigenvalue equation
$Lf\, =- \, \mu f$.

\begin{prop} \label{prop:boundarycondforphiat-1}
The function $\phi_\ga$ satisfies $L\phi_\ga =\mu(\ga)\phi_\ga$ on $\R\setminus\{0\}$.
Moreover, if $f \in V_\mu(\ga)$ is such that $(Lf)(-1)\, =\, \mu(\ga)f(-1)$ and $f(-1)=1$,
then $f = \phi_\ga$ as elements of $\cF_q$.
\end{prop}

\begin{proof}
Proposition \ref{prop:Lphi} gives the first statement.
Lemma \ref{lem:propertiesofphi}(i) shows
\[
\phi_\ga(-q)-\phi_\ga(-1) =  \frac{-ab(1+\ga)}{(1-a)(1-b)}\phi_{\ga/q}(-1;aq,bq;q).
\]
Since $\phi_\ga(-1)=1$, see the remark following \eqref{eq:phi=2phi2A},  we have
\[
(1-a)(1-b)\Big(\phi_\ga(-q)-\phi_\ga(-1) \Big)=-ab(1+\ga)\phi_\ga(-1),
\]
or equivalently, $(L\phi_\ga)(-1) = \mu(\ga)\phi_{\ga}(-1)$ since $B(-1) = 0$.

So $\phi_\ga$ has the properties of $f$ as stated. Now assume that $f$ is a function
satisfying these properties. The values of $f$ on $-q^\N$ are completely determined by the
recurrence relation
\[
\begin{split}
A(-q^k)f(-q^{k+1}) &\, =\, [\mu(\ga)+A(-q^k)+B(-q^k)]f(-q^k)\, -\,  B(-q^k) f(-q^{k-1}),
\qquad k \in \N_{\geq 1},\\
A(-1)f(-q) &\, =\,  [\mu(\ga)+A(-1)]f(-1),
\end{split}
\]
which is just the eigenvalue equation $Lf\, =\, \mu(\ga)f$ on $-q^\N$. Note that
Condition \ref{cond:onparameters} implies that $A(-q^k) \, \not=\, 0$ for $k\in\N$.
So $f=\phi_\ga $ on $-q^\N$, since the solution space is one-dimensional and $f(-1)\, = \, \phi_\ga(-1)$.
In particular, $D(\phi_\ga,f)=0$ on $-q^\N$. By Lemma \ref{lem:dimVmu} we have $D(\phi_\ga,f)=0$ on $-q^{\N+1}\cup tq^\Z$, so that $f=C\,\phi_\ga$ on $-q^{\N+1}\cup tq^\Z$ for some nonzero constant $C$ which we have already determined as $1$. So $f=\phi_\ga$ on $-q^{\N}\cup tq^\Z$.
\end{proof}


\subsection{Asymptotic solutions}\label{ssec:asymptoticsols}

In order to describe the resolvent operator we need to have more
solutions to the eigenvalue equation, especially the ones that
behave nice in the points $tq^k$, $k\to-\infty$.

In order to describe this solution, we first consider another
solution.

\begin{lemma}\label{lem:dimVmu=2} The function $\psi_\ga$ defined by
\begin{equation*}
\psi_\ga(x)= \psi_\ga(x;a,b;q) = \frac{ (qa/b,aq\ga x,-bx;q)_\infty }{ (-qx,-q/b\ga,-q\ga;q)_\infty } \rphis{2}{2}{-a\ga,-ax}{aq\ga x,q/ab}{q,\frac{ q^2}{b}},
\end{equation*}
is  a solution to the eigenvalue equation $L\psi_\ga\, = \, \mu(\ga)\psi_\ga$
on $-q^{\N+1}\cup tq^\Z$, and
$\psi_\ga\in V_{\mu(\ga)}$.
\end{lemma}

The function $\psi_\ga$ is in general not symmetric in $a$ and $b$, hence we
find yet another solution to the eigenvalue equation given by
$\psi_\ga(\cdot;b,a;q)$. Furthermore, since
$(c;q)_\infty \rphis{2}{2}{a,b}{c,d}{q,z}$ is analytic in $c$, we see that
$x\mapsto\psi_\ga(x)$ has simple poles at $-q^{-\N+1}$ and
$\ga \mapsto \psi_\ga(x)$ has poles at $-q^{-\N-1} \cup -b^{-1}q^{1+\N}$.
For generic values ($b\notin q^{2+\N}$) of $b$ these poles are simple.
Also, from the definition we find
$\frac{\te(-bx)}{\te(-b\ga)}\psi_x(\ga)\, = \, \psi_\ga(x)$, so this
solution is almost self-dual. The definition of $\psi_\ga$ is motivated
by the results in \cite[\S 3]{Groe}.

\begin{proof}
Applying \cite[(III.2)]{GaspR} with $(A,B,C,Z) = (-a\ga, -q\ga, aq\ga x, -q/b\ga)$ to obtain \begin{equation} \label{eq:psi}
\psi_\ga(x) = \psi_\ga(x;a,b;q) = \frac{ (aq\ga x,-bx;q)_\infty }{ (-qx,-q\ga;q)_\infty } \rphis{2}{1}{-a\ga, -q\ga}{aq\ga x}{q,-\frac{q}{b\ga}}, \qquad |b\ga|>q.
\end{equation}
Using \eqref{eq:psi} and the $q$-contiguous relation given in the proof of Proposition \ref{prop:Lphi} with the substitution $(A,B,C,Z) \mapsto (-a\ga,-q\ga,aq\ga x, -q/b\ga)$,
we find $L\psi_\ga\, = \, \mu(\ga)\psi_\ga$ after a straightforward calculation and
continuation with respect to $\ga$.

From the ${}_2\vp_2$-expression it is clear that $\lim_{x\to 0}\psi_\ga(x)$ exists.
Using the Leibniz rule for the $q$-derivative, see \cite[Ch.~1]{GaspR}, it suffices
to calculate the $q$-derivatives of $f$ and $g$ in
$\psi_\ga(x) = f(x)g(x)$ with $f(x)=\frac{(-bx;q)_\infty}{(-qx;q)_\infty}$ and
$g$ then given by the definition of $\psi_\ga$. Then the limits of $f$ and $g$ as
$x\to 0$ exist, and the $q$-derivatives $D_qf$, $D_qg$ follow by a straightforward calculation,
and we see that also the limits of $D_qf$ and $D_qg$ exist as
$x\to 0$. It follows that $\psi_\ga\in V_{\mu(\ga)}$.
\end{proof}

It follows that the function defined by
\begin{equation}\label{eq:defPhi1}
\Phi_\ga(x) \, = \, (a,b;q)_\infty\, \phi_\ga(x) \, - \, c(\ga)\, \psi_\ga(x),
\quad c(\ga)\, = \, \frac{\te(-qt, -q\ga, abt\ga)}{\te(aqt\ga, -bt)},
\quad x\in\C\setminus -q^{-\N-1}
\end{equation}
satisfies $\Phi_\ga\in V_{\mu(\ga)}$ for $\ga \notin (at)^{-1}q^\Z\cup -b^{-1}q^{1+\N}$
for generic values of $b$ ($b\notin q^{2+\N}$). For this note that the simple
poles $\ga\in -q^{-\N-1}$ of $\psi_\ga(x)$ are canceled by zeroes of $c(\ga)$.

Next we want to derive an explicit expression for $\Phi_\ga(tq^k)$.
We use \cite[(III.31)]{GaspR} with
$(A,B,C,Z) = (-b\ga,-a\ga,ab\ga x, -1/\ga)$ and multiplying by $(ab\ga x, -1/\ga;q)_\infty$ and using \eqref{eq:phi}, \eqref{eq:psi}, this gives
\[
\begin{split}
(a,b;q)_\infty \phi_\ga(x)\, =\, &e_\ga(x) \frac{\te(-bx)}{\te(-b\ga)}\psi_x(\ga)\, -\, \\
 &\frac{ (-ax, -a\ga,-1/\ga,q^2/ab\ga x;q)_\infty \te(b)}{(-q/bx,-q/b\ga;q)_\infty \te(a \ga x) }\,
\rphis{2}{1}{-q/ax, -q/bx}{q^2/ab\ga x}{q,-\frac{1}{\ga}}
\end{split}
\]
where
$e_\ga(x) \, = \, \frac{ \te(-q\ga,-qx,ab\ga x) }{\te(aq\ga x,-bx) }$.
Now $\frac{\te(-bx)}{\te(-b\ga)}\psi_x(\ga)\, = \, \psi_\ga(x)$ by the
definition of $\psi_\ga$.  Moreover $e_\ga$ is a $q$-periodic function, so that
restricted to $x$ in $tq^\Z$ it gives a constant, which is $c(\ga)$.
From this calculation we find for $|\ga|>1$
\begin{equation}\label{eq:defPhi}
\Phi_\ga(x) \, = \, \frac{ (-ax, -a\ga,-1/\ga,q^2/ab\ga x;q)_\infty \te(b)}{(-q/bx,-q/b\ga;q)_\infty \te(a \ga x) } \rphis{2}{1}{-q/ax, -q/bx}{q^2/ab\ga x}{q;-\frac{1}{\ga}}, \quad x\in tq^\Z.
\end{equation}
This expression can also be used to show that $\Phi_\ga$ is a solution
to the eigenvalue equation by \cite[Exer.~1.12(ii), 1.13]{GaspR}.
By Jackson's  transformation \cite[(III.4)]{GaspR} for $x\in tq^\Z$
\begin{equation}\label{eq:defPhi2}
\Phi_\ga(x)\, = \, \frac{(-ax,-a\ga, q^2/ab\ga x;q)_\infty\, \te(b)}
{(-q/bx, -q/b\ga, a\ga x;q)_\infty}\,
\rphis{2}{2}{-q/ax, -q/a\ga}{q^2/ab\ga x, q/a\ga x}{q, \frac{q}{b\ga x}}.
\end{equation}
Then \eqref{eq:defPhi2} is valid $(-q/bx, -q/b\ga;q)_\infty\, \te(a\ga x)\not=0$.
From \eqref{eq:defPhi2} we get the asymptotic
behaviour.

\begin{lemma}\label{lem:asymptoticPhi} For $\ga\in \C$ so that
$(-q/bx, -q/b\ga;q)_\infty\, \te(a\ga t)\not=0$,
\begin{equation*}
\Phi_\ga(tq^k) \, =\,
(-\ga)^k \frac{ (-a\ga;q)_\infty \te(b,-at)}{(-q/b\ga;q)_\infty \te(at\ga)}
\Big(1+\cO(q^{-k})\Big),  \qquad k\to -\infty.
\end{equation*}
\end{lemma}

\begin{lemma}\label{lem:phiPhiCasoratidet}
\begin{equation*}
D(\Phi_\ga,\phi_\ga)\, =\, -\frac{(1-q)}{t} \frac{ (q/b,-1/\ga,-a\ga;q)_\infty \te(-qt,abt\ga)}{ (a,-q/b\ga;q)_\infty \te(aqt\ga, -bqt)}.
\end{equation*}
\end{lemma}

\begin{proof}
Since $\phi_\ga$ and $\Phi_\ga$ are solutions to the eigenvalue
equation and are elements of $V_{\mu(\ga)}$ the Casorati determinant is constant on
by Lemma \ref{lem:dimVmu}.
We find the value of the determinant by letting $k \to -\infty$ in the explicit expression for $D(\Phi_\ga,\phi_\ga)(t q^{k})$, using Lemmas \ref{lem:propertiesofphi} and \ref{lem:asymptoticPhi} for the asymptotic behaviour of $\phi_\ga$ and $\Phi_\ga$, and \eqref{eq:asymptoticsv} for the behaviour of $v$. We have
\[
\begin{split}
\lim_{k \to -\infty} \Phi_\ga(tq^{k+1}) \phi_\ga(tq^k)  v(tq^k)  \,&=\,  C_\ga \lim_{k \to -\infty}(-\ga)^{k+1}  \Big((-ab\ga t^{-k} q^{-\hf k(k-1)}\Big) \Big( (abt)^k q^{\hf k(k-1)}\Big)\\ \, &=\, -\ga C_\ga,
\end{split}
\]
where
\[
C_\ga = \frac{(1-q)}{t}\frac{ (-1/\ga,-a\ga;q)_\infty \te(-qt,b,-at,ab\ga t)}{ (a,b,-q/b\ga;q)_\infty \te(at\ga,-aqt, -bqt)}.
\]
Similarly
\[
\begin{split}
\lim_{k \to -\infty}  \Phi_\ga(t q^{k})\phi_\ga(t q^{k+1}) v(t q^k)
&\, =\,  C_\ga \lim_{k \to -\infty} (-\ga)^{k} \Big((-ab\ga t)^{-k-1} q^{-\hf k(k+1)}\Big) \Big( (abt)^k q^{\hf k(k-1)}\Big) \\
&\, = \frac{C_\ga}{-ab\ga t} \lim_{k \to -\infty} q^{-k} =0.
\end{split}
\]
Now we obtain
\[
D(\Phi_\ga,\phi_\ga)\, =\, \lim_{k \to -\infty} \Big(\Phi_\ga(t q^{k})\phi_\ga(t q^{k+1})- \Phi_\ga(t q^{k+1}) \phi_\ga(t q^k)\Big) v(t q^k)\, =\, \ga C_\ga,
\]
which proves the result using \eqref{eq:thetaproduct}.
\end{proof}


\subsection{Spectral decomposition}\label{ssec:spectraldecomp}

Now that we have the solutions $\phi_\ga$ and $\Phi_\ga$ available we
can calculate the resolvent operator explicitly. From the resolvent
operator we can calculate explicitly the spectral measure, which
leads to a proof of Theorem \ref{thm:LDselfadjointandspectraldecomp}.

We define the Green kernel $K_\ga(x,y)$ for $x,y \in -q^\N\cup tq^\Z$ by
\[
K_\ga(x,y) =
\begin{cases}
\dfrac{\phi_\ga(x)\Phi_\ga(y)}{D(\ga)}, & x \leq y,\\ \\
\dfrac{\phi_\ga(y)\Phi_\ga(x)}{D(\ga)}, & x > y,
\end{cases}
\]
where $D(\ga) = D(\Phi_\ga,\phi_\ga)$, see Lemma \ref{lem:phiPhiCasoratidet} for the explicit expression. Observe that  for $x,y \in -q^\N\cup tq^\Z$ we have $K_\ga(x,\cdot), K_\ga(\cdot,y) \in \cH_t$. In order to determine the spectral decomposition of $L$ it is important to know where the poles of the Green kernel, considered as a function of $\ga$, are situated.

\begin{lemma} \label{lem:polesK}
Denote $S_{sing}=-q^{\N} \cup (1/abt)q^\Z$ and let $x,y \in -q^\N\cup tq^\Z$. Then $\ga \mapsto K_\ga(x,y)$ has simple poles in $S_{sing}$ and is analytic on $\C\setminus S_{sing}$.
\end{lemma}

\begin{proof}
Fix $x$, $y$ and denote $\cK(\ga) = K_\ga(x,y)$. Recall that $\ga\mapsto\phi_x(\ga)=\phi_\ga(x)$ is an entire function, so the only poles of $\cK$ are poles of $\ga\mapsto\Phi_\ga(x)$ or zeroes of the Casorati determinant $D(\cdot)$
and poles of $D$ may cancel possible poles of $\Phi_\cdot(x)$.
The poles of $\Phi_\cdot(x)$ are $-b^{-1}q^{\Z_{\geq 1}} \cup (at)^{-1}q^\Z$ by the
discussion in Section \ref{ssec:asymptoticsols}.
The poles of $D$ come from the factor $(-q/b\ga;q)_\infty \te(atq\ga)$ in the denominator of $D$. So the poles are simple and they lie in $-b^{-1}q^{\Z_{\geq 1}} \cup (at)^{-1}q^\Z$. Consequently, the poles of $D$ cancel the poles of $\Phi_\cdot(x)$, so the poles of $\Phi_\cdot(x)$ do not contribute to the poles of $\cK$. The zeroes of $D$ are in $-q^{\Z_{\geq 0}} \cup (-1/a)q^{\Z_{\leq 0}} \cup (1/abt)q^{\Z}$, which follows from Lemma \ref{lem:phiPhiCasoratidet}.  We assume that the parameters are generic, so that the
zeroes are all simple.

From \eqref{eq:defPhi2} it follows that for $\ga \in (-1/a)q^{\Z_{\leq 0}}$ the function $\Phi_\ga$ is identically zero on $tq^\Z$, which implies that it is identically zero on $-q^{\N+1}\cup tq^\Z$
by Lemma \ref{lem:dimVmu}, and hence on $-q^{\N}\cup tq^\Z$. So the zeroes of $D$ in $(-1/a)q^{\Z_{\leq 0}}$ are canceled by zeroes of $\ga\mapsto \Phi_\ga(x)$, so these do not contribute to the poles of $\cK$. We conclude that the poles of $\cK$ are the points in the set $S_{sing}$.
\end{proof}

We can describe the resolvent for $(L,\cD)$ with the Green kernel. We introduce the function $\ga\colon \C \to \C$ by $\ga_\la=-(\la/ab+1)$, so that $\mu(\ga_\la)=\la$.

\begin{prop}
Let $\mu \in \C\setminus \R$ and define $R_\mu \colon \cH_t \to \cF_q$ by
\[
(R_\mu f)(y) = \big\langle f, \overline{K_{\ga_\mu}(\cdot,y)} \big\rangle, \qquad f \in \cH_t,\quad
y \in -q^\N\cup tq^\Z,
\]
then $R_\mu$ is the resolvent of $(L,\cD)$.
\end{prop}

\begin{proof}
The proof is the same as the proof of \cite[Prop.~6.1]{KoelS}.
\end{proof}

Using the resolvent $R_\mu$  we can calculate explicitly the spectral measure $E$ for the self-adjoint operator $(L,\cD)$ with the formula,  \cite[Thm.XII.2.10]{DunfS},
\begin{equation} \label{eq:Stieltjes-Perron}
\langle E(\mu_1,\mu_2)f, g \rangle = \lim_{\de \downarrow 0} \lim_{\ep \downarrow 0} \frac{1}{2\pi i} \int_{\mu_1 + \de}^{\mu_2-\de} \Big( \langle R_{\mu+i\ep} f,g \rangle - \langle R_{\mu-i\ep} f,g \rangle \Big) d\mu,
\end{equation}
for $\mu_1<\mu_2$ and $f,g \in \cH$. Using the definition of the Green kernel we have
\begin{equation} \label{eq:intR}
\begin{split}
\langle R_\mu f, g \rangle &= \int_{-1}^{\infty(t)}\int_{-1}^{\infty(t)} f(x)\overline{g(y)} K_{\ga_\mu}(x,y) w(x)w(y) d_qx\, d_qy\\
&= \iint \limits_{x \leq y} \frac{ \phi_{\ga_\mu}(x) \Phi_{\ga_\mu}(y)}{D(\ga_\mu)} \Big(f(x)\overline{g(y)}+f(y)\overline{g(x)}\Big)\big(1-\hf\de_{xy}\big) w(x) w(y) d_qx\,d_qy.
\end{split}
\end{equation}
The Kronecker-delta function $\de_{xy}$ is needed here to prevent the terms on the diagonal $x=y$ from being counted twice. We are now in a position to determine the spectrum and the spectral measure $E$ for the self-adjoint operator $(L,\cD)$.

\begin{prop} \label{prop:spectralmeasure}
The spectrum of the self-adjoint operator $(L,\cD)$ consists of the simple
discrete spectrum $\mu(S_{sing})$ and $\{\mu(0)\}$. Let $\ga \in S_{sing}$ and
assume $\mu_1<\mu_2$ are chosen such that $(\mu_1,\mu_2)\cap \mu(S_{sing}) = \{ \mu(\ga) \}$, then
\[
\langle E(\mu_1,\mu_2) f,g \rangle = ab(a,b;q)_\infty \, \Res{\ga' = \ga}\frac{1}{D(\ga')} \,
\langle f, \phi_\ga \rangle\,  \langle \phi_\ga, g \rangle, \qquad f,g\in\cH_t.
\]
\end{prop}

\begin{proof}
From \eqref{eq:Stieltjes-Perron}, \eqref{eq:intR} and Lemma \ref{lem:polesK} we see that the only contribution to the spectral measure $E$ comes from the poles of $\ga\mapsto K_\ga(x,y)$. Assume $\ga$ is such a pole, i.e., $\ga \in S_{sing}$, and let $\mu_1<\mu_2$ be such that $(\mu_1,\mu_2) \cap \mu(S_{sing}) = \{ \mu(\ga) \}$. Then $\Phi_\ga(x)=(a,b;q)_\infty \phi_\ga(x)$ by \eqref{eq:defPhi1}, since the factor $\te(-q\ga,ab\ga t)$ in front of $\psi_\ga$ is equal to zero. So in this case $\phi_\ga \in \cH_t$ which implies that $\phi_\ga$ is an eigenfunction of $L$, hence $\mu(\ga)$ is in the discrete spectrum of $L$. Now \eqref{eq:Stieltjes-Perron} and \eqref{eq:intR} give
$\langle E(\mu_1,\mu_2) f,g \rangle = \frac{1}{2\pi i} \int_\cC \langle R_\mu f,g\rangle d\mu$,
where $\cC$ is a clockwise oriented, rectifiable contour encircling $\mu(\ga)$ once. Applying Cauchy's theorem we obtain
\[
\begin{split}
\langle& E(\mu_1,\mu_2) f,g \rangle =
ab(a,b;q)_\infty \\ &\times \Res{\ga' = \ga}\frac{1}{D(\ga')}\,  \iint_{x \leq y} \phi_{\ga}(x) \phi_{\ga}(y) \Big(f(x)\overline{g(y)}+f(y)\overline{g(x)}\Big)\big(1-\hf\de_{xy}\big) w(x) w(y) d_qx\,d_qy
\end{split}
\]
The factor $ab$ comes from the substitution $\mu \mapsto -ab(1+\ga)$, where the minus sign is canceled by reversing the orientation of $\cC$. The result now follows from symmetrizing the double $q$-integral.

Since the spectrum is closed, $\mu(0)$ must be in the spectrum of $(L,\cD)$.
\end{proof}

Before proving Corollary \ref{cor:thmLDselfadjointandspectraldecomp}, we calculate
the residue of Proposition \ref{prop:spectralmeasure}.

\begin{lemma} \label{lem:Res(1/D)=w}
For $\ga \in S_{sing}$ we have
\[
(ab) (a,b;q)_\infty\,  \Res{\ga' = \ga}\frac{1}{D(\ga')} \, =\, K_t\, |\ga| w(\ga;a,b;q),
\]
where $w$ is the weight function defined by \eqref{eq:defweight}, and
\[
K_t = K_t(a,b;q) = \frac{1}{1-q} \frac{ (a,b;q)_\infty^2 \, \te(-at,-bt) }{ (q;q)_\infty^2\,  \te(-t,-abt) }.
\]
\end{lemma}

\begin{proof}
From Lemma \ref{lem:phiPhiCasoratidet} we have
\[
\frac{(ab)(a,b;q)_\infty}{D(\ga)} = -\frac{abt}{(1-q)} \frac{ (a,a,b,-q\ga;q)_\infty \te(atq\ga, -b\ga,-bqt)}{ (q/b,-a\ga,-b\ga;q)_\infty \te(abt\ga,-q\ga,-qt)} = C\, f(\ga)w(\ga),
\]
where $C=-\frac{abt}{(1-q)} \frac{ (a,a,b;q)_\infty \te(-bqt)}{ (q/b;q)_\infty \te(-qt)}$ is a constant independent of $\ga$,
and $f(\ga) = \frac{\te(atq\ga, -b\ga)}{\te(abt\ga,-q\ga)}$ is the $q$-periodic function given by
For $\ga \in S_{sing}$ we have
$(ab) (a,b;q)_\infty\,  \Res{\ga' = \ga}\frac{1}{D(\ga')} = C \,w(\ga)\,  \Res{\ga' = \ga}f(\ga')$,
so we only need to calculate the residue of $f$. For $\ga=-q^{k} \in -q^{\Z_{\geq 0}}$ we have
\[
\begin{split}
\Res{\ga' = \ga} \frac{1}{f(\ga')} &= \lim_{z \to -1} (zq^k + q^k)f(zq^k )
= q^{k}\lim_{z \to -1} (z+1)f(z)\\
&= q^k\frac{ \te(-atq,b) }{\te(-abt)} \lim_{z \to -1} \frac{z+1}{\te(-qz)}
=\frac{-q^k \te(-atq,b) }{(q;q)_\infty^2\te(-abt)},
\end{split}
\]
which proves the result for $\ga \in -q^{\Z_{\geq 0}}$. Now assume $\ga =q^k/abt \in (1/abt)q^\Z$, then
\[
\begin{split}
\Res{\ga' = \ga} f(\ga') &= \frac{q^{k}}{abt}\lim_{z \to 1/abt} (abtz-1)f(z)
= \frac{q^{k}}{abt}\frac{ \te(-atq,b) }{\te(-abt)} \lim_{z \to 1/abt} \frac{abtz-1}{\te(abtz)}
= \frac{-q^{k}}{abt}\frac{\te(-atq,b) }{(q;q)_\infty^2\te(-abt)},
\end{split}
\]
which gives the result in this case.
\end{proof}

\begin{proof}[Proof of Corollary \ref{cor:thmLDselfadjointandspectraldecomp}]
Assume $\ga,\la \in S_{sing}$. Since $\phi_\ga$ and $\phi_{\la}$ are eigenfunction of a self-adjoint operator with distinct eigenvalues, they are orthogonal in $\cH_t$. In
case $\la=\ga$, pick $\mu_1<\mu_2$ as in Proposition \ref{prop:spectralmeasure}, so that
\[
\langle \phi_\ga,\phi_\ga \rangle \, = \,
\langle E(\mu_1,\mu_2)\phi_\ga,\phi_\ga\rangle\, =\, K_t\, |\ga|w(\ga)\, \langle \phi_\ga, \phi_\ga \rangle^2,
\]
so that $\langle \phi_\ga,\phi_\ga \rangle = \big( K |\ga|w(\ga) \big)^{-1}$.
After a rewrite the orthogonality relations of
Corollary \ref{cor:thmLDselfadjointandspectraldecomp} follow. We have already remarked,
cf. the discussion following Corollary \ref{cor:thmLDselfadjointandspectraldecomp}, that
the orthogonality relations are self-dual. This in particular implies that
$\{ \phi_\ga(\cdot;a,b;q)\colon \ga \in -q^\N \cup q^\Z/abt\}$ is an
orthogonal basis for $\cH_t$.
\end{proof}

\begin{proof}[Proof of Theorem \ref{thm:LDselfadjointandspectraldecomp}] This follows from Proposition \ref{prop:spectralmeasure}, except for the fact that we have not yet established that $\mu(0)$ is not
in the point spectrum. This follows from Corollary \ref{cor:thmLDselfadjointandspectraldecomp}, since
$\mu(0)\in\si_p(L)$ would imply that the dual orthogonality would not be valid.
\end{proof}


\section{Direct proofs}\label{sec:directproofs}

In this section we give direct proofs of the orthogonality relations of
Corollary \ref{cor:thmLDselfadjointandspectraldecomp} using
transformation and summation for basic hypergeometric series. Since
the polynomial part of Corollary \ref{cor:thmLDselfadjointandspectraldecomp}
is already proved in Proposition \ref{prop:orthogpolMeixner}, it suffices
to deal with the case $\la\in q^\Z/abt$. It should be noted that
direct proof actually extends the orthogonality relations of
Corollary \ref{cor:thmLDselfadjointandspectraldecomp} to a more general
set of parameters, since we do not use the fact that Condition \ref{cond:onparameters}
holds. We only need to assume the condition that the $q$-integrals are well defined.

\begin{proof}[Direct proof of Corollary \ref{cor:thmLDselfadjointandspectraldecomp}]
We need to evaluate the $q$-integrals
\[
I(\ga,\la)\, =\, \frac{1}{1-q}\int_{-1}^{\infty(t)} \phi_\ga(x)\, \phi_\la(x)\, w(x)\, d_q x,
\qquad \ga,\la \in -q^\N\cup q^\Z/abt.
\]
The case $\ga,\la\in -q^\N$ has been proved directly using
\eqref{eq:qMeixnerfunctionpols} and Proposition \ref{prop:orthogpolMeixner}. So
we restrict to the case $I_m(\ga)\, = \, I(\ga, 1/abtq^{m-1})$, $m\in\Z$.
Use \eqref{eq:phi=2phi2A} to write
\begin{equation}\label{eq:dirproof1}
\phi_{1/abtq^{m-1}}(x)\, = \, \sum_{n=0}^\infty \frac{(-q^{1-m}/at, -bx;q)_n}{(q,b;q)_n}
\frac{(q^{1+n-m}x/t;q)_\infty}{(a;q)_\infty}\, (-a)^n q^{\hf n(n-1)}
\end{equation}
and so we have for $p\in \N$
\[
\begin{split}
\int_{-1}^{\infty(t)} (-ax;q)_p\, \phi_{1/abtq^{m-1}}(x)\, w(x)\, d_q x \, = \,
\sum_{n=0}^\infty &\frac{(-q^{1-m}/at;q)_n\, (-a)^n q^{\hf n(n-1)}}{(q,b;q)_n\, (a;q)_\infty}
\\ &\times \int_{-1}^{tq^{m-n}} \frac{(-qx, q^{1+n-m}x/t;q)_\infty}{(-aq^px, -bq^nx;q)_\infty}\, d_qx.
\end{split}
\]
The interchange of summations at zero is no problem because of
\eqref{eq:behaviourwatzero}, and for $x=tq^k$, $k\to\infty$, the term
$(q^{1+n-m}x/t;q)_\infty$ in the weight function gives zero for $n-m+k<0$ and the other terms
yield a term $q^{\hf k(k-1)}$ assuring absolute convergence.
For $n,k \in \N$ and $m \in \Z$ we have
\begin{multline}\label{eq:finitecaseofProp}
\frac{1}{1-q}\int_{-q^{k}}^{tq^{m-n}} \frac{ (-q^{1-k}x,
q^{n-m+1}x/t;q)_\infty }{ (-ax,-bq^n x;q)_\infty}\,  d_q x =\\ \frac{
(q,-abtq^m;q)_\infty
 \te(-tq^{m-n}) }{(a,bq^n, -atq^{m-n}, -btq^m;q)_\infty } \frac{ (a,
bq^n;q)_k }{(-abtq^m;q)_k } (tq^{m-n})^k q^{-\hf k(k-1)},
\end{multline}
see the discussion following Lemma \ref{lem:full2psi2sum} and
\eqref{eq:t-=-1oflemmafull2psi2sum}. Using \eqref{eq:finitecaseofProp} and
straightforward manipulations we find
\[
\begin{split}
\int_{-1}^{\infty(t)} (-ax;q)_p\, \phi_{1/abtq^{m-1}}(x)\, w(x)\, d_q x \, = \,
C\, \rphis{1}{1}{-q^{1-m}/at}{-q^{1-p-m}/at}{q, \frac{1}{q^p}} \, = \,
\frac{C\, (q^{-p};q)_\infty}{(-q^{1-p-m}/at;q)_\infty}
\end{split}
\]
which is zero,
by the summation \cite[(II.5)]{GaspR}. Hence, by
\eqref{eq:qMeixnerfunctionpols} we find $I_m(\ga)=0$ for $\ga\in -q^\N$.

In order to deal with the last part, we start with the $q$-integral
\begin{equation}\label{eq:qintdirectproof}
\int_{-1}^{\infty(t)} \phi_\ga(x)\, (-bx;q)_n (q^{1-n+m}x/t;q)_\infty\, w(x)\, d_qx.
\end{equation}
Inserting \eqref{eq:phi=2phi2} in the form
\[
\phi_\ga(x) = \sum_{k=0}^\infty \frac{(-q^{1-k} x, -1/\ga;q)_k
}{(q,a,b;q)_k} (-ab\ga)^k q^{k(k-1)},
\]
and interchanging summations, which is easily justified since the summation
corresponding to the $q$-integral $\int_0^{\infty(t)}$ is a unilateral
sum in this case, we find, using \eqref{eq:finitecaseofProp}, that
\eqref{eq:qintdirectproof} equals
\begin{equation}\label{eq:qintdirectproof2}
\frac{(q,-abtq^m;q)_\infty\, \te(-tq^{m-n})}{(a,bq^n,-atq^{m-n},-btq^m;q)_\infty}
\, \rphis{2}{2}{-1/\ga, bq^n}{b, -abtq^m}{q, ab\ga tq^{m-n}}.
\end{equation}
Now put $\ga= 1/abtq^{r-1}$, $r\in\Z$, so that
using \eqref{eq:dirproof1}, \eqref{eq:qintdirectproof}, \eqref{eq:qintdirectproof2} we find
\begin{equation*}
\begin{split}
I_m(\frac{1}{abtq^{r-1}})\, = \, &
\sum_{n=0}^\infty \frac{(-q^{1-m}/at;q)_n\, (-a)^n q^{\hf n(n-1)}}{(q,b;q)_n\, (a;q)_\infty}
\frac{(q,-abtq^m;q)_\infty\, \te(-tq^{m-n})}{(a,bq^n,-atq^{m-n},-btq^m;q)_\infty} \\
& \qquad \times  \, \rphis{2}{2}{-abtq^{r-1}, bq^n}{b, -abtq^m}{q, q^{1-r+m-n}}.
\end{split}
\end{equation*}
where we interchanging summation and $q$-integration, which can be justified
using the estimates in Lemma \ref{lem:asymptoticPhi} and $\phi_{1/abtq^{r-1}} =
(a,b;q)_\infty^{-1}\, \Phi_{1/abtq^{r-1}}$, see \eqref{eq:defPhi}.

We can transform the ${}_2\vp_2$-series
using \cite[(III.23)]{GaspR} to a terminating ${}_2\vp_1$-series.
Using elementary rewritings and \eqref{eq:thetaproduct} we find
\begin{equation*}
I_m(\frac{1}{abtq^{r-1}}) =
\frac{(q, q^{1+m-r};q)_\infty\, \te(-tq^m)}{(a,a,b,-atq^m,-btq^m;q)_\infty}
\sum_{n=0}^\infty \frac{(-1)^n q^{\hf n(n-1)}}{(q;q)_n}
\, \rphis{2}{1}{q^{-n}, -abtq^{r-1}}{b}{q, q^{1-r+m}}
\end{equation*}
which is zero for $r>m$. In case $r=m$ the ${}_2\vp_1$-series is summable
by the $q$-Vandermonde summation \cite[(II.6)]{GaspR}, and the resulting series
is a summable ${}_1\vp_1$-sum by \cite[(II.5)]{GaspR}. This gives
\[
I_m(\frac{1}{abtq^{m-1}})\, =\, \left(\frac{(q;q)_\infty}{(a,b;q)_\infty}\right)^2
\frac{\te(-tq^m)\, (-abtq^{m-1};q)_\infty}{(-atq^m,-btq^m;q)_\infty}
\]
and collecting the results proves Corollary \ref{cor:thmLDselfadjointandspectraldecomp}
using \eqref{eq:thetaproduct}.
\end{proof}

Another direct proof of the orthogonality is based on Lemma \ref{lem:trinprod(f,g)=D(f,g)}  and
Proposition \ref{prop:Lphi}, and showing that the right hand side of
Lemma \ref{lem:trinprod(f,g)=D(f,g)} vanishes for $f=\phi_\la$, $g=\phi_\ga$. This
gives $\langle\phi_\ga,\phi_\la \rangle=0$ for $\ga\not=\la$.

It is also possible to prove $I_m(\ga)=0$ for arbitrary $\ga$ satisfying
$|\ga| < |1/abtq^{m-1}|$.


\section{Orthogonality relations on $\R$}\label{sec:orthorelsonR}
In Proposition \ref{prop:orthogpolMeixner} we have obtained orthogonality relations for the $q$-Meixner polynomials $m_n$ with respect to the indefinite inner product
\[
(f,g) = \int_{\infty(t_-)}^{\infty(t_+)} f(x) g(x) w(x) d_qx,
\]
where $w$ is the weight functions defined by \eqref{eq:defweight}. In this section we show that there are more functions orthogonal with respect to this indefinite inner product.
Here we assume still that $t_-<0$ and $t_+>0$, but we don't require any other
conditions on the parameters $a$ and $b$ except that $t_\pm q^\Z$ are not zeroes
of the denominator of $w$.

\subsection{Direct proofs on $\R$}\label{ssec:directproofsonR}
We consider the function $\Phi_\ga$ as defined by \eqref{eq:defPhi1} with $t=t_+$. First we study the case $\ga= -q^{1+n}/a$ with $n \in \N$. In this case it follows from \eqref{eq:defPhi} and duality that $\Phi_\ga$ can be expressed in terms of a terminating $_2\varphi_1$-series:
\begin{equation} \label{eq:Phiterminatingseries}
\begin{split}
\Phi_{-q^{1+n}/a}(x)&=\frac{ (-ax, q^{n+1},-q^{1-n}/bx;q)_\infty \te(b)}{(-qx,-q/bx,aq^{-n}/b;q)_\infty } (qx)^n q^{\hf n(n-1)} \rphis{2}{1}{q^{-n}, aq^{-n}/b}{-q^{1-n}/bx}{q;-\frac{1}{x}}\\
&= \frac{ (-ax, q^{n+1};q)_\infty \te(b)}{(-qx,a/b;q)_\infty } \rphis{2}{1}{q^{-n}, -bx}{qb/a}{q,\frac{q^{2+n}}{a}},
\end{split}
\end{equation}
which remains valid for $x\in\R$.
The second expression follows from reversing the order of summation in the first $_2\varphi_1$-series. Comparing this with the polynomials $m_n$ defined in Proposition \ref{prop:orthogpolMeixner} we see that
\[
\Phi_{-q^{1+n}/a}(x)=\frac{(q^2/a;q)_n (-ax, q^{1+n};q)_\infty \te(b) }{ (-qx,a/b;q)_\infty} m_n(ax/q;q^2/a,qb/a;q)
\]
From the orthogonality relations for $m_n$ given in Proposition \ref{prop:orthogpolMeixner} we can now derive orthogonality relations for $\Phi_{-q^{1+n}/a}$.

\begin{prop} \label{prop:orthogonalityPhionR}
For $m,n \in \N$,
\[
(\Phi_{-q^{1+m}/a}, \Phi_{-q^{1+n}/a})  =
\de_{mn} (1-q)\frac{ q^{-n} (q^2/a;q)_n }{ (q, qb/a;q)_n } \frac{ (q;q)_\infty^3\, \te(b)^2\,  \te(qbt_-t_+,t_-/t_+, q^2/a)}{(q^2/a, a/b;q)_\infty\, \te(-qt_-,-qt_+,-bt_-,-bt_+)}.
\]
\end{prop}

\begin{proof}
From the substitution rule
$\int_0^{\infty(z)} f(x)\, d_qx \, =\,  \alpha \int_0^{\infty(z/\alpha)} f(\alpha y) \, d_q y$,
$\al \neq 0$,
we find, using the substitution $y=ax/q$ and Proposition \ref{prop:orthogpolMeixner},
\[
\begin{split}
\int_{\infty(t_-)}^{\infty(t_+)}& \Phi_{-q^{1+m}/a}(x)\Phi_{-q^{1+n}/a}(x) \frac{ (-qx;q)_\infty }{(-ax,-bx;q)_\infty} d_q x \\
=&\,\frac{(q^2/a;q)_m\, (q^2/a;q)_n\, (q^{1+m},q^{1+n};q)_\infty \, \te(b)^2 }{ (a/b;q)_\infty^2} \\
& \times\,  \frac{q}{a} \int_{\infty(at_-/q)}^{\infty(at_+/q)} m_m(y;q^2/a,qb/a;q) m_n(y;q^2/a,qb/a;q) \frac{ (-qy;q)_\infty }{(-q^2y/a,-qby/a;q)_\infty} d_qy \\
=&\,\de_{mn} \frac{q}{a}\left(\frac{(q^2/a;q)_n  (q^{1+n};q)_\infty \te(b) }{ (a/b;q)_\infty}\right)^2 h_n(q^2/a,qb/a)\, I(q^2/a, qb/a;at_-/q,at_+/q).
\end{split}
\]
This proves the result.
\end{proof}

The weight function $w(\,\cdot\,;a,b;q)$ in Proposition \ref{prop:orthogonalityPhionR} is symmetric in $a,b$, but the asymptotic solution $\Phi_\ga(\,\cdot\,;a,b;q)$ is not. Therefore, interchanging $a$ and $b$ in Proposition \ref{prop:orthogonalityPhionR} gives orthogonality relations with respect to $(\cdot,\cdot)$ for yet another set of functions. We define
\[
\begin{split}
\Phi_\ga^\dag(x)\, =\, \Phi_\ga^\dag(x;a,b;q)\,  =\,  K(x)\Phi_\ga(x),\quad
K(x)\, =\, \frac{ \te(-bx,-b\ga,a, a\ga x) }{ \te(-ax, -a\ga, b, b\ga x)},
\end{split}
\]
then it is easily verified using \eqref{eq:defPhi} that $\Phi_\ga^\dag(x;a,b;q) = \Phi_\ga(x;b,a;q)$.
$K$ is a $q$-periodic function, so that $K$ is the constant function $K(t_\pm)$ on $t_\pm q^\Z$, and $\Phi_\ga^\dag$ is actually a multiple of $\Phi_\ga$ on $t_\pm q^\Z$. Now, interchanging $a$ and $b$ in Proposition \ref{prop:orthogonalityPhionR} gives us the following orthogonality relations for $\Phi_{-q^{1+n}/b}^\dag$.

\begin{cor} \label{cor:orthogonalityPhidagonR}
For $m,n \in \N$,
\[
(\Phi_{-q^{1+m}/b}^\dag, \Phi_{-q^{1+n}/b}^\dag)   =
\de_{mn} \, (1-q)\, \frac{ q^{-n} (q^2/b;q)_n }{ (q, qa/b;q)_n } \frac{ (q;q)_\infty^3\,  \te(a)^2 \te(qat_-t_+,t_-/t_+, q^2/b)}{(q^2/b, b/a;q)_\infty\,  \te(-qt_-,-qt_+,-at_-,-at_+)}.
\]
\end{cor}

Next we show that the sets $\{m_n \mid n \in \N\}$, $\{ \Phi_{-q^{1+n}/a} \mid n \in \N\}$ and $\{ \Phi_{-q^{1+n}/b}^\dag \mid n \in \N\}$ are orthogonal to each other.

\begin{prop}
For $n,k \in \N$,
\[
( \Phi_{-q^{1+n}/a},  m_k) = ( \Phi_{-q^{1+n}/b}^\dag,  m_k)  =
 (\Phi_{-q^{1+n}/a} , \Phi_{-q^{1+k}/b}^\dag) = 0.
\]
\end{prop}

\begin{proof}
First observe that by \eqref{eq:c=0oflemmafull2psi2sum}
\[
\int_{\infty(t_-)}^{\infty(t_+)} \frac{ (-qx;q)_k }{(-cx;q)_\infty} d_q x=\int_{\infty(t_-)}^{\infty(t_+)} \frac{ (-qx;q)_\infty }{(-q^{1+k}x,-cx;q)_\infty} d_q x =0,\qquad k \in \N,
\]
for any $c \not\in -t_\pm^{-1} q^\Z$, because of the factor $\te(q^{1+k})$ in Lemma
\ref{lem:full2psi2sum}.
This result implies that
\begin{equation} \label{eq:int=0}
\int_{\infty(t_-)}^{\infty(t_+)} \frac{ p(x) }{(-cx;q)_\infty} d_q x=0,
\end{equation}
for any polynomial $p$. From \eqref{eq:Phiterminatingseries} it follows that
\[
\Phi_{-q^{1+n}/a}(x)m_k(x)\,  =\,  \frac{ (-ax;q)_\infty }{(-qx;q)_\infty}p(x),
\]
with $p$ a polynomial of degree $k+n$. Now we find
\[
\int_{\infty(t_-)}^{\infty(t_+)} \Phi_{-q^{1+n}/a}(x) \,  m_k(x)\,  \frac{ (-qx;q)_\infty }{(-ax,-bx;q)_\infty}\,  d_q x\,
 = \,  \int_{\infty(t_-)}^{\infty(t_+)} \frac{ p(x) }{(-bx;q)_\infty}\,  d_qx\,  =\, 0,
\]
which proves the first identity. The second identity follows from interchanging $a$ and $b$ in the first one.
For the third identity we note that
\[
\Phi_{-q^{1+n}/a}(x)\,  \Phi_{-q^{1+k}/b}^\dag(x)\,  = \,
\frac{ (-ax,-bx;q)_\infty }{ (-qx;q)_\infty^2 } \, p(x),
\]
with $p$ a polynomial of degree $n+k$. Then applying \eqref{eq:int=0} gives us
\[
\int_{\infty(t_-)}^{\infty(t_+)} \Phi_{-q^{1+n}/a}(x)\,  \Phi_{-q^{1+k}/b}^\dag(x)\,  \frac{ (-qx;q)_\infty }{(-ax,-bx;q)_\infty}\,  d_qx\,  =\,
\int_{\infty(t_-)}^{\infty(t_+)}  \frac{ p(x) }{(-qx;q)_\infty}\,  d_qx\,  =\, 0. \qedhere
\]
\end{proof}

\subsection{Indirect proofs using spectral analytic ideas}\label{ssec:resolventproofsonR}
In this section we prove orthogonality relations with respect to $(\,\cdot\,,\,\cdot\,)$ for $\Phi_\ga$ with $\ga \in (-1/abt_-t_+)q^\Z$. The proves are inspired by the spectral analytic method from Section \ref{sec:spectraldecomposition}, but we don't use spectral theory for self-adjoint operators here.

First we need an analogue of Lemma \ref{lem:dimVmu}.

\begin{lemma}\label{lem:dimVmuR}
For $\mu\in \C$ we define
\[
V_\mu \, = \, \{f:t_-q^\Z \cup t_+q^\Z \to \C \mid Lf(x)\, = \, \mu f,\  f(0^+)\,=\, f(0^-), \ f'(0^+)\,=\, f'(0^-)\}.
\]
Then $\dim V_\mu\leq 2$. Moreover, for $f_1,f_2\in V_\mu$  the Casorati
determinant $D(f_1,f_2)$ is constant as a function on $t_-q^\Z \cup t_+q^\Z$. In case $\dim V_\mu = 2$, the restriction operators from $V_\mu$ to the spaces $\{f:t_-q^\Z \cup t_+q^\Z \to \C \mid Lf(x)\, = \, \mu f(x)\ \text{for}\
x\in t_\pm q^\Z\}$ are bijections.
\end{lemma}

The Casorati determinant $D(\cdot,\cdot)$ is defined by \eqref{eq:D(f,g)}. Similar as in Section \ref{sec:spectraldecomposition} it follows that the functions $\phi_\ga$ and $\psi_\ga$ are elements of $V_{\mu(\ga)}$.
We define $\Phi_\ga^+ \in V_{\mu(\ga)}$, respectively $\Phi_\ga^-\in V_{\mu(\ga)}$, as \eqref{eq:defPhi1} with $t=t_+$, respectively $t=t_-$. Explicitly,
\begin{equation}\label{eq:defPhi+-}
\Phi_\ga^\pm(x) \, = \, (a,b;q)_\infty\, \phi_\ga(x) \, - \, c_\pm(\ga)\, \psi_\ga(x),
\quad c_\pm(\ga)\, = \, \frac{\te(-qt_\pm, -q\ga, abt_\pm\ga)}{\te(aqt_\pm \ga, -bt_\pm)}.
\end{equation}
Note that the function $\Phi_\ga$ defined in \eqref{eq:defPhi}  is the function $\Phi_\ga^+$. As in Lemma \ref{lem:asymptoticPhi} we find asymptotic behaviour
\begin{equation} \label{eq:asymptoticsPhi+-}
\Phi_\ga^\pm(t_\pm q^k) = (-\ga)^k \frac{(-a\ga;q)_\infty \te(b, -at_\pm)}{-q/b\ga;q)_\infty \te(at_\pm \ga)}\left(1+\mathcal O(q^{-k})\right), \qquad k \to -\infty,
\end{equation}
so that $\Phi_\ga^+$ is square $q$-integrable on $t_+q^\Z$ with respect to $w$, and similarly $\Phi_\ga^-$ is square $q$-integrable on $t_-q^\Z$.

\begin{lemma} \label{lem:Phi+Phi-Casdet}
\[
D(\Phi^+_\ga,\Phi_\ga^-) = -qbt_+(1-q) \frac{ (-1/\ga,-a\ga;q)_\infty \te(b,-abqt_+t_-\ga,-a\ga, b/q, t_-/t_+) }{ (-q/b\ga;q)_\infty \te(aqt_-\ga,aqt_+\ga, -bt_-, -bt_+)}.
\]
\end{lemma}

\begin{proof}
From the expansion \eqref{eq:defPhi+-} of $\Phi_\ga^-$ in terms of $\phi_\ga$ and $\psi_\ga$ it follows that
\[
D(\Phi^+_\ga,\Phi_\ga^-) = (a,b;q)_\infty D(\Phi_\ga^+, \phi_\ga) - c_-(\ga) D(\Phi_\ga^+, \psi_\ga).
\]
The Casorati determinant $D(\Phi_\ga^+, \phi_\ga)$ is given in Lemma \ref{lem:phiPhiCasoratidet} with $t=t_+$, and from expanding $\psi_\ga$ in terms of $\phi_\ga$ and $\Phi_\ga^+$, see \eqref{eq:defPhi+-}, we obtain
\[
D(\Phi_\ga^+, \psi_\ga) = \frac{(a,b;q)_\infty }{c_+(\ga)} D(\Phi_\ga^+,\phi_\ga).
\]
This gives us the following explicit expression
\[
D(\Phi^+_\ga,\Phi_\ga^-) =(1-q) \frac{ (-1/\ga,-a\ga;q)_\infty \te(b) }{ (-q/b\ga;q)_\infty } \left(t_-^{-1}\frac{ \te(-qt_-,abt_-\ga)}{ \te(aqt_-\ga, -bqt_-)}-t_+^{-1}\frac{ \te(-qt_+,abt_+\ga)}{ \te(aqt_+\ga, -bqt_+)} \right).
\]
Now we apply the identity, see \cite[Exer.2.16(i)]{GaspR},
\[
\te(xv,x/v,yw,y/w)-\te(xw,x/w,yv,y/v) = \frac{y}{v} \te(xy,x/y,vw,v/w)
\]
with
\begin{align*}
x &= ie^{i\be/2} a\ga \sqrt{ -|b|qt_+t_-}, & y&= -ie^{i\be/2}\sqrt{ -|b|qt_+t_-}, \\
v&= ie^{i\be/2}\sqrt{ -\frac{|b|t_-}{qt_+}}, & w &= -ie^{i\be/2}\sqrt{ -\frac{|b|t_+}{qt_-}},
\end{align*}
where $b=|b|e^{i\be}$, then the result follows.
\end{proof}

For $\ga \in (-1/abt_-t_+)q^\Z$ the Casorati determinant in Lemma \ref{lem:Phi+Phi-Casdet} equal zero, hence $\Phi_\ga^+ = k\Phi_\ga^-$ on $t_-q^\Z\cup t_+q^\Z$ for some nonzero constant $k$, so in this case the inner product $(\Phi_\ga^+, \Phi_\ga^+)$ is finite,
since summability at zero is valid as well.
Using \eqref{eq:defPhi+-} we can check that $k=1$, and therefore we omit the superscript $+$ or $-$ in this case.

Let us write $\ga_n = -q^n/abt_-t_+$ for $n \in \Z$. We are going to determine orthogonality relations for the functions $\Phi_{\ga_n}$. We start with an easy result.

\begin{prop} \label{prop:orthogonalityPhiganothers}
Let $at_-t_+,bt_-t_+,abt_-t_+ \not\in q^\Z$, then for $n \in \Z$, $k \in \N$,
\[
(\Phi_{\ga_n},m_k) =  (\Phi_{\ga_n},\Phi_{-q^{1+k}/a})=(\Phi_{\ga_n},\Phi_{-q^{1+k}/b}^\dag)=0.
\]
\end{prop}

\begin{proof}
First note that by Lemmas \ref{lem:phiPhiCasoratidet} and \ref{lem:Phi+Phi-Casdet} $m_k =\phi_{-q^k} = k_+ \Phi_{-q^{k}}^+=k_-\Phi_{-q^{-k}}^-$ for certain nonzero constants $k_\pm$. Using \eqref{eq:asymptoticsPhi+-} one can now check that the integrals $(L\Phi_{\ga_n},m_k)$, $(L\Phi_{\ga_n},\Phi_{-q^{1+k}/a})$, $(L\Phi_{\ga_n},\Phi_{-q^{1+k}/b}^\dag)$ are finite. All functions in the inner products in the proposition are eigenfunctions of the difference operator $L$ for mutually different eigenvalues. The orthogonality relations follow using the fact that
$L$ is symmetric with respect to $(\cdot,\cdot)$, which is proved completely
analogously as in Section \ref{sec:spectraldecomposition}. For this we also
note that all solutions satisfy $f(0^+)=f(0^-)$ and $f'(0^+)=f'(0^-)$.
\end{proof}

Next we consider the inner products $(\Phi_{\ga_m},\Phi_{\ga_n})$, $m,n \in \Z$. We will prove the following result.

\begin{prop} \label{prop:orthogonalityPhiganm}
For $m,n \in \Z$,
\[
\begin{split}
(\Phi_{\ga_m},\Phi_{\ga_n}) =&\, \de_{mn} \frac{(1-q)t_+}{q}  \frac{(b/q)^n   (q/abt_-t_+;q)_n }{ (1/at_-t_+,1/bt_-t_+;q)_n }\\
& \times \frac{ (q;q)^2 (abt_-t_+, 1/bt_-t_+;q)_\infty }{ (aqt_-t_+;q)_\infty } \frac{ \te(b)^2 \te(bt_-t_+,t_-/t_+) }{ \te(-bt_-,-bt_+)^2}.
\end{split}
\]
\end{prop}

For $m\neq n$ this is proved in the same way as Proposition \ref{prop:orthogonalityPhiganothers}. For $m=n$ the proof of Proposition \ref{prop:orthogonalityPhiganm} basically mimics the proof for the orthogonality relations from Corollary \ref{cor:thmLDselfadjointandspectraldecomp} given in Section \ref{sec:spectraldecomposition}, but without using any theory for self-adjoint operators on Hilbert spaces.\\

We define for $x,y \in t_-q^\Z \cup t_+q^\Z$
\[
K_\ga(x,y) =
\begin{cases}
\dfrac{\Phi_\ga^-(x)\Phi_\ga^+(y)}{D(\ga)}, & x \leq y,\\ \\
\dfrac{\Phi_\ga^-(y)\Phi_\ga^+(x)}{D(\ga)}, & x > y,
\end{cases}
\]
where $D(\ga) = D(\Phi_\ga^+,\Phi^-_\ga)$. The explicit expression for $D$ is given in Lemma \ref{lem:Phi+Phi-Casdet}. For $x,y \in t_-q^\Z\cup t_+q^\Z$, the functions $K_\ga(x,\,\cdot\,)$ and $K_\ga(\,\cdot\,,y)$ are square integrable on $t_-q^\Z\cup t_+q^\Z$ with respect to $w$. We need to know the location of the poles of $\ga \mapsto K_\ga(x,y)$
and for this we assume that the parameters are chosen generically, i.e.
$a,b,b/a, abt_-t_+, at_-t_+, bt_-t_+ \notin q^\Z$.

\begin{lemma}
For $x,y \in t_-q^\Z\cup t_+q^\Z$ the function $\ga \mapsto K_\ga(x,y)$ has simple poles in
\[
S_{sing}= -q^{\N} \cup (-q/a)q^{\N} \cup  (-q/b)q^{\N} \cup (-1/abt_-t_+) q^\Z,
\]
and is analytic on $\C\setminus S_{sing}$.
\end{lemma}

\begin{proof}
Fix $x,y \in t_-q^\Z\cup t_+q^\Z$ and denote $K_\ga(x,y)$ by $\cK(\ga)$. Possible contributions to the poles of $\cK$ come from the poles of $\ga\mapsto \Phi_\ga^\pm(x)$, so possible simple poles are in $(1/at_-)q^\Z \cup (1/at_+)q^\Z$, and possible double poles are in $(-q/b)q^{\N}$. But $D$ also has simple poles in $(1/at_-)q^\Z \cup (1/at_+)q^\Z$, so $\cK$ has no poles in this set. Furthermore, $D$ has simple poles in $(-q/b)q^{\N}$, so $\cK$ also has simple poles in this set.

Other possible poles of $\cK$ come from the zeroes of $D$, so possible simple poles in $-q^{\N} \cup (-q/a)q^{\N} \cup  (-q/b)q^{\N} \cup (-1/abt_-t_+) q^\Z$, and possible double poles in $-(1/a)q^{-\N}$. From \eqref{eq:defPhi2} we see that both
$\ga\mapsto \Phi_\ga^+(t_+q^k)$ and $\ga\mapsto \Phi_\ga^-(t_-q^k)$ have simple zeroes in $-(1/a)q^{-\N}$. By Lemma \ref{lem:dimVmuR} the functions $\ga\mapsto \Phi_\ga^+(t_-q^k)$ and $\ga\mapsto \Phi_\ga^-(t_+q^k)$ are also zero on $-(1/a)q^{-\N}$. So we find that $\cK$ has only simple poles in $S_{sing}$, and is analytic on $\C\setminus S_{sing}$.
\end{proof}

The main step for the proof on Proposition \ref{prop:orthogonalityPhiganm} is to prove the following result.

\begin{lemma} \label{lem:RmuPhi}
Let $n \in \Z$. Define for $\mu \in \C$ such that $\ga_\mu=-(\mu/ab+1) \not\in S_{sing}$ the function $R_\mu\Phi_{\ga_n}$ on $t_-q^\Z \cup t_+q^\Z$ by
\[
(R_\mu\Phi_{\ga_n})(y) = \Big(\Phi_{\ga_n}, K_{\ga_\mu}(\,\cdot\,,y)\Big), \qquad y \in t_-q^\Z \cup t_+q^\Z,
\]
then
\[
(\Phi_{\ga_n},\Phi_{\ga_n}) = \frac{-1}{2\pi i} \int_{\mathcal C} (R_{\mu}\Phi_{\ga_n} , \Phi_{\ga_n}) d\mu.
\]
where $\mathcal C$ is a counterclockwise oriented, rectifiable contour that encircles $\mu(\ga_n)$ once, and no other points in $S_{sing}$.
\end{lemma}

\begin{proof}
First of all, we have $(L-\mu) (R_\mu\Phi_{\ga_n})  = \Phi_{\ga_n}$ as an identity on $t_-q^\Z\cup t_+q^\Z$. This is proved similarly as \cite[Prop. 6.1]{KoelS}. Now we find
\[
(\Phi_{\ga_n},\Phi_{\ga_n}) = \big( (L-\mu) (R_\mu\Phi_{\ga_n}) ,\Phi_{\ga_n}\big)= \big( R_\mu\Phi_{\ga_n}, (L-\mu)(\Phi_{\ga_n})\big)= (\ga_n-\mu)\big(R_\mu\Phi_{\ga_n},\Phi_{\ga_n}\big).
\]
This gives us
\[
(\Phi_{\ga_n},\Phi_{\ga_n}) = \frac{1}{2\pi i} \int_{\mathcal C} \frac{ (\Phi_{\ga_n},\Phi_{\ga_n}) }{ \mu-\ga_n} d\mu= \frac{-1}{2\pi i} \int_{\mathcal C} (R_\mu\Phi_{\ga_n},\Phi_{\ga_n})d\mu.
\]
where $\mathcal C$ is a contour as described in the lemma.
\end{proof}
\begin{proof}[Proof of Proposition \ref{prop:orthogonalityPhiganm}]
We use Lemma \ref{lem:RmuPhi}, where we write out the inner product inside the contour integral as a double $q$-integral, and we apply dominated convergence, then
\[
\begin{split}
(\Phi_{\ga_n},\Phi_{\ga_n})& =
ab \Res{\ga' = \ga_n}\frac{1}{D(\ga')} \\
&\times \iint \limits_{x \leq y} \Phi_{\ga_n}(x) \Phi_{\ga_n}(y) \Big(\Phi_{\ga_n}(x)\Phi_{\ga_n}(y)+\Phi_{\ga_n}(y)\Phi_{\ga_n}(x)\Big)\big(1-\hf\de_{xy}\big) w(x) w(y) d_qx\,d_qy.
\end{split}
\]
Symmetrizing the double $q$-integral then gives
\[
(\Phi_{\ga_n},\Phi_{\ga_n}) =
ab \Res{\ga' = \ga_n}\frac{1}{D(\ga')}(\Phi_{\ga_n},\Phi_{\ga_n})^2.
\]
Proposition \ref{prop:orthogonalityPhiganm} now follows after evaluating the residue.
\end{proof}
The orthogonality relations from Section \ref{ssec:directproofsonR} can be proved in the same way as Proposition \ref{prop:orthogonalityPhiganm}.


\section{Limit transitions}\label{sec:limittransitions}

Indeterminate moment problems in the $q$-Askey scheme have been
studied by Christiansen \cite{Chri}, and quite a few of the
cases in \cite{Chri} have been studied using related techniques. We
are inspired by the scheme \cite[p.24]{Chri} in discussing the
limit transitions.

\subsection{Limit from continuous
dual $q^{-1}$-Hahn polynomials}
The study of the big $q$-Jacobi function transform \cite{KoelS} leads to
an explicit orthogonality measure for the continuous
dual $q^{-1}$-Hahn polynomials, which are at the top of
the indeterminate moment problems in \cite{Chri}. In the
big $q$-Jacobi functions
\[
\tilde\phi_{\tilde\ga} (x;\tilde a, \tilde b, \tilde c;q)\, = \,
\rphis{3}{2}{\tilde a \tilde\ga, \tilde a/\tilde\ga, -1/x}{\tilde a\tilde b, \tilde a\tilde c}{q, -\tilde b\tilde c x}
\]
we substitute $\tilde \ga = -\tilde a \ga$, $\tilde a\tilde b = a$,
$\tilde a\tilde c= b$ and we let $\tilde b\to 0$. Then the big $q$-Jacobi function
tends to the $q$-Meixner function \eqref{eq:phi=2phi2}. Also, after multiplying
by $\tilde b \tilde c$, the operator \cite[(2.2-3)]{KoelS}
tends to $L$ defined by \eqref{eq:defL}. In this formal limit for the eigenvalue
equation we see that the continuous spectrum in \cite{KoelS} shrinks to zero,
the finite part of the discrete spectrum of  \cite{KoelS} tends to the spectrum $-q^\N$ and
the semifinite discrete part of the spectrum of \cite{KoelS} tends to doubly infinite
discrete spectrum $q^\Z/abt$ for the $q$-Meixner functions. Note that the
limit case discussed in this paper is self-dual, whereas the big $q$-Jacobi transform
is not self-dual.
The conditions
\cite[(2.1)]{KoelS} on  the parameters for the big $q$-Jacobi functions
lead to $0<a,b<1$, that is Condition \ref{cond:onparameters}(ii).

\subsection{Limit to $q$-Laguerre polynomials}
The $q$-Laguerre polynomials are a well-known set of orthogonal polynomials
corresponding to an indeterminate moment problem, see \cite{CiccKK}, \cite{GaspR}
and references given there. One of the standard orthogonality relations
is related to Ramanujan's ${}_1\psi_1$-summation, which can be viewed
as a $q$-integral over $tq^\Z$. Replacing $x$ and $a$ in
$xc$ and $a/c$ and letting $c\to \infty$ we find
that \eqref{eq:phi=2phi1} with $x\leftrightarrow\ga$ by self-duality tends
to
${}_1\vp_1(-1/\ga; b;q, abx\ga)$
which are the functions studied in \cite{CiccKK}. Then in the limit the support
of the orthogonality measure reduces to a constant times $q^\Z$, and the
structure of the spectrum remains unchanged, so $-q^\N$ corresponds to
the $q$-Laguerre polynomials and the constant times $q^\Z$ corresponds to
the big $q$-Bessel functions of \cite{CiccKK}. The limit in the
eigenvalue equation reduces to the operator studied in \cite{CiccKK}.
A classical limit then also leads from the $q$-Laguerre polynomials back
to the Stieltjes-Wigert polynomials, see \cite{Chri}.

\subsection{$q$-Charlier polynomials and Al-Salam--Carlitz II polynomials}
For the indeterminate moment problems for the
$q$-Charlier polynomials and Al-Salam--Carlitz II polynomials
have not be studied by this method, so that the formal limit transition is
not known. We refer to \cite{Chri} for more information and
references.


\end{document}